\documentclass{amsart}
\usepackage{amsfonts, amsmath, amssymb, mathrsfs}
\usepackage{amsthm, url}
\usepackage{graphicx,epstopdf}
\usepackage{mathdots}
\usepackage{subfig}
\usepackage{color}
\usepackage{bbm}
\usepackage{hyperref}
\usepackage[fit]{truncate}
\usepackage{paralist}

\theoremstyle{plain}
\newtheorem{Theorem}{Thm}[section]
\newtheorem{Thm}[Theorem]{Theorem}
\newtheorem{Lem}[Theorem]{Lemma}
\newtheorem{Cor}[Theorem]{Corollary}
\newtheorem{Prop}[Theorem]{Proposition}
\newtheorem{Ques}[Theorem]{Questions}
\newtheorem*{Thm*}{Theorem}
\newtheorem*{Lem*}{Lemma}
\newtheorem*{Rem*}{Remark}

\theoremstyle{definition}
\newtheorem{Def}[Theorem]{Definition}
\newtheorem{Exm}[Theorem]{Example}
\newtheorem{Rem}[Theorem]{Remark}

\newtheorem{theorem}{Theorem}

\theoremstyle{definition}

\newtheorem{exampleth}[theorem]{Example}

\DeclareMathOperator{\codim}{codim}

\DeclareMathOperator{\conv}{conv}

\DeclareMathOperator{\Mat}{Mat}
\DeclareMathOperator{\Sym}{Sym}
\DeclareMathOperator{\Herm}{Herm}

\DeclareMathOperator\dd{d}
\DeclareMathOperator\im{im}
\DeclareMathOperator\rank{rank}
\DeclareMathOperator\disc{Disc}

\newcommand\wh{\widehat}
\newcommand\ol{\overline}

\newcommand\scp[1]{\langle #1\rangle}

\newcommand{\R}{\mathbb{R}}
\newcommand{\Q}{\mathbb{Q}}

\newcommand{\Z}{\mathbb{Z}}

\renewcommand\P{\mathbb{P}}
\newcommand{\C}{\mathbb{C}}
\newcommand{\V}{\mathcal{V}}

\newcommand\x{\underline x}
\newcommand\m{m}
\DeclareMathOperator\new{Newt}

\DeclareMathOperator{\G}{\mathcal{G}}
\DeclareMathOperator{\GS}{\mathcal{G}^+}
\DeclareMathOperator{\h}{\mathcal{H}}
\DeclareMathOperator{\HS}{\mathcal{H}^+}

\renewcommand\Re{{\rm Re}}
\renewcommand\Im{{\rm Im}}

\title{Gram Spectrahedra}
\author{Lynn Chua}
\address{University of California, Berkeley, CA, USA}
\email{chualynn@berkeley.edu }
\author{Daniel Plaumann}
\address{Technische Universit\"at Dortmund, Dortmund, Germany }
\email{Daniel.Plaumann@math.tu-dortmund.de}
\author{Rainer Sinn}
\address{
Georgia Institute of Technology, Atlanta, GA, USA }
\email{rsinn3@math.gatech.edu}
\author{Cynthia Vinzant}
\address{
North Carolina State University, Raleigh, NC, USA }
\email{clvinzan@ncsu.edu}

\begin{document}
\maketitle

\begin{abstract}
  Representations of nonnegative polynomials as sums of squares are
  central to real algebraic geometry and the subject of active
  research. The sum-of-squares representations of a given polynomial
  are parametrized by the convex body of positive semidefinite Gram
  matrices, called the Gram spectrahedron. This is a fundamental
  object in polynomial optimization and convex algebraic geometry. We
  summarize results on sums of squares that fit naturally into the
  context of Gram spectrahedra, present some new results, and highlight
  related open questions. We discuss sum-of-squares representations
  of minimal length and relate them to Hermitian 
  Gram spectrahedra and point evaluations on toric varieties.
\end{abstract}

\section*{Introduction}
The relationship between nonnegative polynomials and sums of squares is a
beautiful subject in real algebraic geometry. It was greatly influenced in recent years by connections to
polynomial optimization through the theory of moments (see
\cite{BPTMR3075433, Marshall08, Sc09}). A polynomial $f\in \R[\x]$ is
\emph{nonnegative} if $f(x)\geq 0$ for all $x\in \R^n$ and is a
\emph{sum of squares} if $f = p_1^2 + \hdots + p_r^2$ for some
$p_1, \hdots, p_r$ in $\R[\x]$. Clearly, a sum of squares is
nonnegative, but not all nonnegative polynomials are sums of squares.
The study of the relationship between these two notions goes back to
Hilbert \cite{HilbertMR1510517}. Determining the nonnegativity of a
polynomial in more than one variable is computationally difficult in general, 
already for general polynomials of degree four. Writing a polynomial as a sums of squares, on the
other hand, provides a computationally tractable certificate for
nonnegativity. Testing whether a polynomial is a sum of squares is
equivalent to testing the feasibility of a \emph{semidefinite
 program}. In fact, the set of sum-of-squares representations of a
fixed polynomial $f\in \R[\x]$ is naturally written as a
\emph{spectrahedron}, i.e.~the intersection of the cone of
positive semidefinite matrices with an affine-linear space. This
spectrahedron is the \emph{Gram spectrahedron} of $f$ and is the
central object of this paper.

Representations of a polynomial as a sum of squares are far from
unique. For example,
$x^2 + y^2 = (\cos(\theta) x - \sin(\theta) y)^2 + (\sin(\theta) x+
\cos(\theta) y)^2$
for any $\theta \in [0,2\pi]$. More generally, the orthogonal group
in dimension $r$ acts on representations of a polynomial as a sum of
$r$ squares. It was noted in \cite{CLR} that \emph{Gram matrices}
give natural representatives for each orbit of sum-of-squares
representations under this action. The Gram spectrahedron of a
polynomial $f$ is the set of all its positive semidefinite Gram matrices, a parameter space for the sum-of-squares representations of $f$.
So $f$ is a sum of squares if and only if its Gram spectrahedron
is nonempty.
More precisely, a polynomial $f$ has a representation as a sum of $r$ squares if and only its Gram spectrahedron contains a matrix of rank at most $r$.

The aim of this paper is to survey the literature on sum-of-squares
representations and Gram spectrahedra, present some new results, 
and highlight related open questions. We are
particularly interested in the convex algebro-geometric properties of
Gram spectrahedra. This includes the basic geometric properties, the
ranks of extreme points, and the corresponding algebraic degree of
optimization.

One fundamental question is: What is the minimum rank of a positive
semidefinite Gram matrix of a polynomial? We consider this question 
for generic polynomials with a fixed \emph{Newton polytope}.
We describe a recent result due to Blekherman, Smith, and Velasco, 
which produces an explicit characterization of the Newton 
polytopes for which all nonnegative polynomials are sums of squares 
(Theorem~\ref{thm:NewtonPolytopes}). In these
cases, the minimum rank of a Gram matrix is known
\cite{BPSV2016arXiv160604387B}. In this paper, we extend this analysis
to determine the shortest sum-of-squares representations over
varieties of \emph{almost} minimal degree
(Theorem~\ref{Thm:ACMalmost}).

The algebraic degree of optimization over a spectrahedron is a measure
of the size of the field extension needed to write the point
maximizing a linear function. These degrees were studied and computed
for general spectrahedra in \cite{algDegGH} and \cite{algDegNRS}. For
Gram spectrahedra, these degrees are not well-understood. We discuss
some small cases. The algebraic degree of optimization over Gram
spectrahedra of ternary quartics was studied numerically in
\cite{PSV}. For Gram spectrahedra of binary sextics, this algebraic
degree is small and we can write the optimal point in radicals. This
involves using the classical theory of Kummer surfaces to write down
an explicit formula for the dual surface to the boundary of the Gram
spectrahedron.

The paper is organized as follows. In Section~\ref{sec:Gram}, we
introduce basic definitions and properties of Gram spectrahedra, 
their relation to sum-of-squares representations, and requisite field extensions.
Connections with toric varieties and the consequences 
of the results in \cite{BSVMR3486176} for
nonnegative polynomials with a given Newton polytope 
are described in Section~\ref{sec:toric}. 
Section~\ref{sec:Ranks} is devoted to the ranks of extremal matrices
in Gram spectrahedra, with a special focus on those of minimum rank. An original result determines the shortest sum-of-squares representations
of quadratic forms on varieties of almost minimal degree. Gram spectrahedra of binary
forms and ternary quartics are explored in greater depth in
Section~\ref{sec:Examples}. Section~\ref{sec:Hermitian} describes a
Hermitian analogue of the Gram spectrahedron that characterizes
Hermitian sum-of-squares representations. We conclude in
Section~\ref{sec:Biquadratic} with a description of open
questions connecting the sum-of-squares length of a specific family of
polynomials to questions in topology and the computational complexity of the permanent.

\medskip \textbf{Acknowledgements.} We would especially like to thank
Claus Scheiderer and Bernd Sturmfels, from whom we have learned much
of the theory presented here. We are grateful to Greg Blekherman,
Thorsten Mayer, Mateusz Micha\l{}ek, Bruce Reznick, Emmanuel Tsukerman, and Mauricio Velasco for helpful
discussions. We also thank Friedrich Knop for answering a question on
toric varieties on \textit{Mathoverflow}. Lynn Chua was supported by a
UC Berkeley Graduate Fellowship and the Max Planck Institute for
Mathematics in the Sciences, Leipzig. Daniel Plaumann received
financial support from the Zukunftskolleg of the University of
Konstanz and DFG grant PL 549/3-1. Rainer Sinn was partially supported
by NSF grant DMS-0757212.  Cynthia Vinzant received support from the
NSF (DMS-1204447) and the FRPD program at N.C. State.

\section{Gram Matrices and Sums of Squares}\label{sec:Gram}
In this section, we outline the Gram matrix method introduced in \cite{CLR}.
Given a real polynomial $f\in\R[\x]$ in $n$ variables $\x=(x_1,\dots,x_n)$,
the Newton polytope of $f$ is the convex hull of all exponents of
monomials occurring in $f$ with non-zero coefficient and is denoted
$\new(f)$. We fix a polytope $P\subset\R^n$ with vertices in $\Z_{\geq
 0}^n$ and write $\R[\x]_{2P}$ for the vector space of all
polynomials $f\in\R[\x]$ whose Newton polytope is contained in $2P$.
We write $\Sigma_{2P}$ for the convex cone of all polynomials in $\R[\x]_{2P}$ that are sums of squares.
When $P$ is the scaled simplex, $d \Delta_{n-1} = \{ \alpha \in \R_{\geq 0}^n : \sum_i \alpha_i =d\}$, we will 
write $\R[\x]_{d\Delta} = \R[\x]_d$.  
The study of sums of squares via their Gram matrices and Newton polytopes was pioneered by 
Choi, Lam, and Reznick in \cite{CLR}, who observed the following. 

\begin{Prop}[{see \cite[Theorem~3.5]{CLR}}]\label{Prop:newton}
If $\new(p_1^2+\cdots+p_r^2)$ is contained in $2P$ for some $p_1,\dots,p_r\in\R[\x]$, 
then $\new(p_i)$ is contained in $P$ for all $i=1,\dots,r$.
\end{Prop}

\begin{proof}
Let $f = p_1^2 + \dots + p_r^2$ and assume for contradiction that $\new(f)\subset 2P$ but
$\new(p_1)\nsubseteq P$. We can further assume that $\new(p_1)$ has a vertex $\alpha$ that
is not in $P$ and is an extreme point of the convex hull $Q$ of $\bigcup_{i=1}^r \new(p_i)$. Then the coefficient of $x^{2 \alpha}$ in $f$ must be $0$ because $2\alpha$ does not lie in $2P\supset \new(f)$. Since $\alpha$ is an extreme point of $Q$, there is a linear functional $\ell$ such that $\ell(\alpha) = -1$ and $\ell(\beta)>-1$ for all extreme points $\beta$ of $Q$. Therefore, $\ell(2 \alpha) = -2$ and $\ell(\beta + \beta') > -2$ for any extreme points $\beta,\beta'$ of $Q$ such that $\beta\neq \alpha$ or $\beta'\neq \alpha$. Since every extreme point of $2Q$ is the sum of two extreme points of $Q$, this means that the coefficient of $x^{2\alpha}$ is the sum of squares $\sum_{i=1}^r (a_\alpha^i)^2$ of the coefficients $a_\alpha^i$ of $x^\alpha$ in $p_i$. Since $a_\alpha^1\neq 0$, this coefficient of $p_1^2 + \dots + p_r^2$ is non-zero, a contradiction.
\end{proof}

We fix an order of the monomials in $x_1,\dots,x_n$, whose exponent vectors are lattice points in $P$, and write $\m_P$ for the vector of these monomials in the fixed order. Let $N$ be the length of this vector, i.e.~the number of lattice points in $P$.
Write $\Sym_N$ for the vector space of real symmetric $N\times N$ matrices and $\Sym_N^+$ 
for the cone of positive semidefinite matrices in $\Sym_N$. 

\begin{Prop}[{see \cite[Theorem~2.4]{CLR}}]\label{Prop:GramMethod}
A polynomial $f\in\R[\x]_{2P}$ is a sum of squares if
and only if there is a matrix $A \in \Sym_N^+$ such that
\[
f = \m_P^t A \m_P.
\]
\end{Prop}

\begin{proof}
If $f = p_1^2 + \dots + p_r^2$, then by Proposition~\ref{Prop:newton}, 
$\new(p_i)\subseteq P$ for each $i=1, \hdots, r$. Let $c_i$ be the vector of coefficients
of $p_i$ such that $p_i = c_i^t \m_P$. Let $C$ be the matrix whose
$j$-th row is the vector $c_j^t$. Then $A = C^t C = \sum_i c_ic_i^t$ is a positive
semidefinite Gram matrix of $f$. Conversely, let $A$ be a positive
semidefinite Gram matrix of $f$ of rank $r$. Then it follows from the
principal axes theorem that there exists a
factorization $A = C^t C$, where $C$ is an $r\times N$ matrix
. This shows that $f = p_1^2 + \dots +p_r^2$, where the coefficient vector of $p_j$ is the $j$-th row of $C$.
\end{proof}

Proposition~\ref{Prop:GramMethod} is the basis for the computation of
sum-of-squares decompositions using semidefinite
programming (see for example \cite[Ch.~10]{Marshall08}). 
The set of all matrices satisfying the conditions of
Proposition~\ref{Prop:GramMethod} make up a spectrahedron.

\begin{Def}\label{Def:GramSpec}
Fix a polytope $P\subset \R^n$ with vertices in $\Z_{\geq 0}^n$ and let $N$ be the number of lattice points in $P$. 
Let $f\in\R[\x]_{2P}$ be a polynomial with $\new(f)\subseteq 2P$.
\begin{enumerate}[(a)]
\item Every matrix $A\in \Sym_N$ such that $f = \m_P^t A \m_P$ holds is called a \emph{Gram matrix} of $f$. We write $\G(f)$ for the affine-linear space of all Gram matrices of $f$.
\item The \emph{Gram spectrahedron} $\GS(f)$ of $f$ is the intersection of the cone of positive semidefinite matrices with the affine-linear space of Gram matrices of $f$, i.e.
\[
\GS(f) \ \ = \ \ \G(f)\cap\Sym_N^+  \ \ = \ \  \{A \in \Sym_N^+\colon \m_P^t A \m_P = f \}.
\]
\item The \emph{length} of a sum-of-squares representation
  $f=p_1^2+\cdots+p_r^2$ is $r$, the number of summands. The
  \emph{length} (or \emph{sum-of-squares length}) of the polynomial $f$ is
  the shortest length of any sum-of-squares representation of $f$. If $f$
  is not a sum of squares, its length is defined to be infinity.
\end{enumerate}
\end{Def}

By Proposition~\ref{Prop:GramMethod}, the length of a polynomial $f$ is
equal to the minimum rank of any matrix in $\GS(f)$. Indeed, a
matrix $A\in\GS(f)$ of rank $r$
gives rise to a representation of $f$
of length $r$.
Conversely, the proof shows that a representation of $f$ as a sum of $r$ squares 
leads to an explicit Gram matrix of rank at most $r$.

The sum-of-squares representation corresponding to a positive semidefinite
Gram matrix of a polynomial is not unique but rather depends on the
choice of a decomposition. The following lemma is helpful in stating
this precisely.

\begin{Lem}\label{Lem:OrthogonalEquivalence}
  Let $N,r\ge 0$ and let $B,C\in\Mat_{r\times N}(\R)$. Then
  $B^tB=C^tC$ if and only if there exists an orthogonal $r\times
  r$-matrix $U$ such that $C=UB$. 
\end{Lem}

\begin{proof}
  Suppose that $B^tB=C^tC$. Let $b_1,\dots,b_N$ and $c_1,\dots,c_N$ be
  the column vectors of $B$  and $C$ (resp.) and let
  $V_B \subset\R^r$ and $V_C\subset\R^r$ be the column spans. By
  hypothesis, the columns of $B$ and $C$ have the same pairwise inner
  products $\scp{b_i,b_j}=\scp{c_i,c_j}$, $i,j=1,\dots,N$. This
  implies that there exists a linear isometry $\phi\colon V_B\to V_C$
  with $\phi(b_i)=c_i$ for all $i$. Any extension of $\phi$ to $\R^r$
  yields the desired orthogonal matrix $U$. The converse is obvious.
\end{proof}

We say that two sum-of-squares representations of the same polynomial
are \emph{equivalent} if they give the same Gram
matrix. By the lemma above, two representations
$f=p_1^2+\cdots+p_r^2=q_1^2+\cdots+q_r^2$ of
the same length are equivalent if and only if there exists an
orthogonal $(r\times r)$-matrix $U$ such that
\[
(q_1,\dots,q_r)^t=U\cdot (p_1,\dots,p_r)^t.
\]
Two representations of different length may be equivalent if there are linear relations among the summands. 
This equivalence was introduced in \cite[Proposition~2.10]{CLR}.

\begin{Lem}\label{lem:generaltop}
Let $P\subset\R^n$ be a lattice polytope whose vertices lie in $\Z_{\geq 0}^n$.
\begin{enumerate}[(a)]
\item For every polynomial $f\in \R[\x]_{2P}$, the Gram spectrahedron $\GS(f)$ is compact.
\item (see \cite[Proposition~5.5]{CLR}) The Gram spectrahedron $\GS(f)$ contains a positive definite matrix if and only if the polynomial $f$ lies in the interior of the cone of sums of squares in $\R[\x]_{2P}$.
\end{enumerate}
\end{Lem}

\begin{proof}
The Gram spectrahedron $\GS(f)$ is closed, as it is
the intersection of two closed sets. Suppose $\GS(f)$ was
unbounded. Then there exists a positive semidefinite Gram matrix $A_n$
of $f$ with norm $n$ (in the Frobenius norm), for all sufficiently
large $n\in \Z_{\geq 0}$. Thus the matrix $\frac1n A_n$ is a positive
semidefinite Gram matrix of $\frac1n f$ with norm $1$. By compactness
of the unit sphere in $\Sym_N$, the sequence $(A_n)_n$ has a
convergent subsequence. The limit $A$ of this subsequence is a positive
semidefinite matrix representing the zero polynomial. But the only
such matrix is the zero matrix,
contradicting the fact that $A$ has norm $1$.

Part (b) follows from the general fact that linear maps preserve
(relative) interiors of convex sets. For the sake of completeness, we
include a proof of this statement. Let $C\subset\R^m$ be a convex set
with non-empty interior and let $\phi\colon\R^m\to\R^n$ be a
surjective linear map. Since $\phi$ is open, we have
$\phi({\rm int}(C))\subset{\rm int}(\phi(C))$. Conversely, let
$x\in\R^n$ be a point with $\phi^{-1}(x)\cap{\rm
  int}(C)=\emptyset$.
By weak separation of convex sets, there exists a linear functional
$\lambda\colon\R^m\to\R$ with $\lambda|_{\phi^{-1}(x)}=0$ and
$\lambda|_{{\rm int}(C)}>0$. Since ${\rm int}(C)\neq\emptyset$, this
implies $\lambda|_C\ge 0$ and since $\lambda$ is constant on the
fibers of $\phi$, it induces a functional $\lambda'\colon\R^n\to\R$
via $\lambda'(\phi(y))=\lambda(y)$ for $y\in\R^m$. Thus
$\lambda'(x)=0$ and $\lambda'|_{\phi(C)}\ge 0$, so that $x$ is not an
interior point of $\phi(C)$.

Applying this general fact to the Gram map $\Sym_N\to\R[x]_{2P}$,
$A\mapsto m_P^t A m_P$, which maps $\Sym_N^+$ to the cone of sums of squares in $\R[x]_{2P}$,
we obtain (b).
\end{proof}

From a computational point of view, it is also interesting to consider rational sum-of-squares representations of rational polynomials.
\begin{Lem}[{see also \cite[Section 3.6]{BPTMR3075433}}]\label{Lem:RationalGramMatrix}
A polynomial $f \in \Q[\x]_{2P}$ with rational coefficients is a sum
of squares of polynomials with rational coefficients if and only if
the Gram spectrahedron $\GS(f)$ contains a matrix with rational entries.
\end{Lem}

\begin{proof}
If $f = p_1^2 + \dots + p_r^2$, where the polynomials $p_i\in \Q[\x]_P$, 
then the corresponding positive semidefinite 
Gram matrix of $f$ has rational entries. Conversely, suppose $A$ is a
positive semidefinite Gram matrix of $f$. We will factor it more
carefully than before: First, diagonalize $A$ over $\Q$ as a
quadratic form, i.e.~find an invertible matrix $U$ such that $U^t A U
= D$. Then $D$ is a diagonal matrix with rational entries and with the
same signature as $A$. The entries of $D$ are nonnegative. By
the theorem of Lagrange, we can write every diagonal entry of $D$ as a
sum of four squares, so we can write $D = D_1^2 + D_2^2 + D_3^2 +
D_4^2$, where the matrices $D_1,D_2,D_3,D_4$ are diagonal with
rational entries. In this way, we obtain a rational sum-of-squares
representation of $f$ of length $4\cdot{\rm rank}(A)$, namely
\[
f = \sum_{j=1}^4 (U^{-1}\m_P)^tD_j^t D_j (U^{-1} \m_p) = \m_P^t \left(\sum_{j=1}^4 (D_j U^{-1})^tD_j U^{-1}\right) \m_P.\qedhere
\]
\end{proof}

\begin{Rem}
  We may wish to consider sum-of-squares representations of polynomials over other fields, 
  like number fields. Lemma~\ref{Lem:OrthogonalEquivalence} holds over any ordered
  field, but Lemma~\ref{Lem:RationalGramMatrix} does not. This is
  because a field may admit more than one ordering, whereas the
  ordering of $\Q$ and $\R$ is unique. For example, the polynomial
  $x^2+\sqrt{2}$ is not a sum of squares in $\Q(\sqrt{2})[x]$, even
  though it possesses a positive semidefinite Gram matrix with entries
  in $\Q(\sqrt{2})\subset\R$, because $\sqrt{2}$ is not a sum of squares in the
  field $\Q(\sqrt{2})$.  The correct generalization of
  Lemma~\ref{Lem:RationalGramMatrix} to any ordered
  field therefore must use a stronger notion of positivity for
  symmetric matrices (namely, positivity under any ordering of the
  field), see \cite[Chapter VIII]{LaMR2104929} or \cite[Chapter
  3]{PrDeMR1829790}.
\end{Rem}

We may ask whether the condition in Lemma \ref{Lem:RationalGramMatrix}
holds for any polynomial $f$ with rational
coefficients. This holds if $f \in\Q[\x]_{2P}$ lies in
the interior of $\Sigma_{2P}$, since the Gram spectrahedron of $f$ will be full dimensional in 
$\Sym_N$ by Lemma~\ref{lem:generaltop}(b) and $\Sym_N(\Q)$ is dense in $\Sym_N(\R)$. 
On the boundary of the cone of sums of squares, this is no longer the case.  There are
in fact rational polynomials that are sums of squares over $\R$ but
not over $\Q$. The existence of such polynomials was shown by
Scheiderer, along with explicit examples, answering a question of
Sturmfels that had been open for some time. 
As observed below, the example of Scheiderer is minimal, as 
any such an example must have degree at least $4$ in at least $3$ variables.

\begin{Rem}
For $P = \Delta_{n-1}$, any $f\in\Q[\x]_{2P}$ is a quadratic form, which has a uniquely determined Gram matrix 
with rational entries. So it is positive semidefinite if and only if $f$ is a sum of squares of linear forms with rational coefficients.

Furthermore, if $f\in\Q[x]$ is a nonnegative univariate polynomial, 
then $f$ is a sum of squares of rational
polynomials. If $f$ is positive, it has a rational Gram matrix
by the arguments above. If $f$ has a zero $z\in \R$, then the minimal
polynomial $p$ of $z$ over $\Q$ must divide $f$ to an even power
because $p\in\Q[x]$ is separable (in particular, $p'(z)\neq 0$). Since
$p^2(h_1^2 + \dots + h_r^2) = (ph_1)^2 + \dots + (ph_r)^2$, we can
reduce to the case that $f$ is strictly positive. 
A representation of $f$ as a sum of rational squares can be found algorithmically as explained in \cite[Kapitel 2]{Schweighofer}.
\end{Rem}

\begin{Exm}{\cite[Example 2.8]{ScheidererRational}}.
The ternary quartic with rational coefficients 
\[
f = x^4 + xy^3 + y^4 - 3x^2yz - 4xy^2z + 2x^2z^2 + xz^3 + yz^3 + z^4
\]
is positive semidefinite. Its Gram spectrahedron is a line segment that is the convex hull of two representations of $f$ as a sum of two squares over the number field $\Q(\sqrt{-\beta})$ of degree $6$, where the minimal polynomial of $\beta$ over $\Q$ is $t^3-4t-1$. The representations as sums of two squares over $\Q(\sqrt{-\beta})$ are given via the identity
\[
4 f = \left( 2x^2 + \beta y^2 - yz + (2 + \frac1\beta) z^2\right)^2 + (\sqrt{-\beta})^2 \left(2xy - \frac1\beta y^2 + \frac2\beta xz + \beta yz - z^2 \right)^2
\]
by the two real embeddings of $\Q(\sqrt{-\beta})$. Scheiderer constructs this ternary quartic $f$ as a product of four linear forms over $\Q(\alpha)$, where $\alpha^4 - \alpha + 1=0$; namely
\[
f = (x + \alpha y + \alpha^2 z) (x + \ol{\alpha}y + \ol{\alpha}^2 z) (x + \zeta y + \zeta^2 z) (x + \ol{\zeta} y + \ol{\zeta}^2 z),
\]
where $\alpha, \ol{\alpha}, \zeta, \ol{\zeta}$ are the four roots of $t^4 - t + 1$, which come in complex conjugate pairs and the bar denotes complex conjugation. The minimal polynomial $t^3 - 4t -1$ is the resolvent cubic of the quartic minimal polynomial of $\alpha$. The factorization of $f$ over $\C$ shows that $f$ has two real zeros $(\alpha \ol{\alpha}: -\alpha-\ol{\alpha}:1)$ and $(\zeta \ol{\zeta}: -\zeta - \ol{\zeta}:1)$, which are singular points of the curve $\V(f)\subset\P^2$, and four complex singularities.

More generally, Scheiderer shows that every ternary quartic
$f\in\Q[x,y,z]$ that is a sum of squares over $\R$ but not over $\Q$
factors as $f = \ell_1 \ell_2 \ell_3 \ell_4$ in $\C[x,y,z]$ and the absolute Galois group of $\Q$ acts as the full symmetric group or the alternating group on four letters on this set of four lines; see \cite[Theorem 4.1]{ScheidererRational}.
\end{Exm}
For an approach to this example using symbolic computations for
linear matrix inequalities, see \cite[Section
5.2]{HNS2015arXiv150803715H}. The construction using products of
linear forms can be generalized to higher degree and more variables,
see \cite[Theorem 2.6]{ScheidererRational}. But this will
always result in forms with zeros, leaving the following as an open problem.
\begin{Ques}[{\cite[5.1]{ScheidererRational}}]
Does there exist  $f\in\Q[\x]_{2P}$ that is a sum of squares in
$\R[\x]_{2P}$, strictly positive on $\R$, but not a sum of squares over $\Q[\x]_P$?
\end{Ques}

Such an example must lie in the interior of the cone of
nonnegative polynomials and on the boundary of the cone $\Sigma_{2P}$
of sums of squares by Lemma \ref{lem:generaltop}(b). There
are explicit constructions of such polynomials in the literature, e.g.~\cite{BlMR2904568} and \cite{Reznick}, 
but it is not clear how to keep track of rationality.

\section{Connections to Toric Geometry}\label{sec:toric}
The Gram matrix method for fixed Newton polytopes can be understood through
toric geometry, see \cite[Section 6]{BSVMR3486176}. 
An excellent reference for toric geometry is \cite{CLSMR2810322}.

We begin by listing the Newton polytopes for which every nonnegative polynomial is a sum of squares.
This uses the classification of all projective varieties $X\subset\P^{N-1}$ such
that every nonnegative quadratic form on $X$ is a sum of squares in
the homogeneous coordinate ring of $X$, carried out by
Blekherman-Smith-Velasco in \cite{BSVMR3486176}. We translate their
result into a statement purely about Newton polytopes (see also \cite[Theorem~6.3]{BSVMR3486176}), using toric geometry and the classification of varieties of
minimal degree by del Pezzo and Bertini (see \cite{EisenbudHarris} for a modern presentation).
\begin{Thm}[{\cite[Theorems~1.1 and 6.3]{BSVMR3486176}}]\label{thm:NewtonPolytopes}
Let $P\subset\R^n$ be a lattice polytope and suppose that $P\cap \Z^n$ generates $\Z^n$ as a group. Further suppose that every nonnegative polynomial $f\in \R[\x]_{2P}$ with Newton polytope $2P$ is a sum of squares. Then the lattice polytope $P$ is, up to translation and an automorphism of the lattice, contained in one of the following polytopes.
\begin{enumerate}
\item The $m$-dimensional standard simplex $\conv\{0,e_1,\dots,e_m\}\subset \R^m$, where \\
$e_1,\dots,e_m$ is the standard basis of $\R^m$.
\item The Cayley polytope of $m$ line segments $[0,d_i]$ ($d_i\in \Z_{\geq 0}$, $i =1,\dots,m$):
\[
\conv \left\{ ([0,d_1]\times e_1) \cup ([0,d_2]\times e_2) \cup \dots \cup ([0,d_m]\times e_m)\right\}\subset \R\times\R^{m}.
\]
\item The scaled $2$-simplex $\conv\{(0,0),(2,0),(0,2)\}\subset\R^2$.
\item The free sum $\conv\{(Q\times\{0\}) \cup (\{0\}\times\Delta_{n-1})\}\subset\R^m\times \R^n$, where $Q\subset\R^m$ is one of the preceding polytopes (1)-(3) and $\Delta_{n-1} = \conv\{e_1,\dots,e_n\}\subset\R^n$.
\end{enumerate}
\end{Thm}

\begin{proof}
Suppose the lattice polytope $P\subset\R^n$ has dimension $n$, contains $N$ lattice points, and $P\cap \Z^n$ generates $\Z^n$. List the lattice points of $P$ as $m_1,\dots,m_N$. Then $P$ defines an $n$-dimensional projective toric variety $X_P\subset\P^{N-1}$, which is the Zariski closure of the image of the map
\begin{equation}\label{eq:toric}
(\C^\ast)^n \to \P^{N-1}, \ \ x\mapsto (x^{m_1}:x^{m_2}:\dots:x^{m_N}).
\end{equation}
A polynomial $f\in\R[\x]_{2P}$ is the restriction of a quadratic form
to $X_P$ and the affine space of Gram matrices of $f$ is the set of
all quadratic forms in $N$ variables whose restriction to $X_P$ give
$f$. By \cite[Theorem 1.1]{BSVMR3486176}, every nonnegative quadratic
form on $X_P$ is a sum of squares if and only if $X_P$ is a variety of
minimal degree. By the classification of varieties of minimal degree
\cite{EisenbudHarris}, the toric variety $X_P$ has to be projectively
equivalent to one of the toric varieties given by the polytopes listed
(1)-(4). Note that rational normal scrolls
correspond to the Cayley polytopes in case (2), the Veronese surface
in $\P^5$ corresponds to the scaled $2$-simplex in case (3), and a
cone over a smooth variety of minimal degree corresponds to the
polytope in case (4). Two toric varieties $X_P\subset\P^{N-1}$ and
$X_Q\subset \P^{N-1}$ given by spanning lattice polytopes
$P,Q\subset\R^n$ are projectively equivalent if and only if the
polytope $P$ is obtained from $Q$ via an affine-linear isomorphism of
$\Z^n$, that is a composition of a translation and a lattice
automorphism, which implies the claim. While this last fact is known to experts in toric geometry, we are unable to find a suitable reference, and a proof would take us too far outside the scope of this paper.
\end{proof}

\begin{Exm}
Consider the polytope $P = \conv\{(0,0),(1,0),(2,1)\}$, which is the image of $\Delta_2$ under the lattice isomorphism that maps $(1,0)$ to $(1,0)$ and $(0,1)$ to $(2,1)$.
Theorem~\ref{thm:NewtonPolytopes} implies that every nonnegative polynomial with support in $2P = \conv\{(0,0),(2,0),(4,2)\}$ is a sum of squares, i.e. polynomials of the form
\[
p = a + bx+cx^2+dx^2y+ex^3y+fx^4y^2.
\]
Indeed, assuming $f=1$, it can be written as $p = \left(x^2 y + \frac12 (d+e x)\right)^2 - \frac14 \disc_y(p)$ and $p$ is nonnegative if and only if $-\disc_y(p)$ is a sum of two squares.
\end{Exm}

We have discussed all toric cases of varieties of minimal degree. There is one more case in the classification, that of quadratic hypersurfaces, which is intimately related to the S-procedure in the optimization literature.
\begin{Rem}\label{rem:quadratichypersurface}
Let $X = \V(Q)\subset\P^n$ be a quadratic hypersurface defined by a quadratic form $Q\in\R[x_0,\dots,x_n]_2$. If a quadratic form $f$ is nonnegative at all points $x\in\R^{n+1}$ with $Q(x) = 0$, then $f$ is a sum of squares of linear forms modulo $Q$ because $X$ is a variety of minimal degree, see \cite[Theorem 1.1]{BSVMR3486176}. In other words, there is a $\lambda\in\R$ such that $f + \lambda Q$ is a positive semidefinite matrix. In analogy with the toric case, we call the set $\{\lambda\in\R\colon f + \lambda Q \text{ is positive semidefinite}\}$ the Gram spectrahedron of $f$ modulo $Q$. This Gram spectrahedron is a point if and only if $f$ and $Q$ have a common real zero in $\P^n$. If $f$ is strictly positive on $X(\R)$, then the Gram spectrahedron modulo $Q$ is a line segment. Its interior points are positive definite matrices and, for generic $f$, the two boundary points have rank $n$. So a generic quadratic form $f$, which is positive on $X(\R)$, has two representations as a sum of $n$ squares of linear forms modulo $Q$ and no shorter representations.

The hyperbolicity of the determinant of the symmetric matrix implies that, for generic positive $f$ modulo $Q$, all $n+1$ rank-$n$ matrices representing $f$ are real.
\end{Rem}

This is the only case in which generic Gram spectrahedra are line segments.
\begin{Prop}\label{prop:quadratichypersurf}
Let $X\subset \P^n$ be a variety defined by quadrics. Suppose that $q\in \R[X]_2$ is an interior point of the cone of sums of squares in $\R[X]_2$. 
The Gram spectrahedron $\GS(q)$ is a line segment if and only if $X$ is a quadratic hypersurface.
\end{Prop}
\begin{proof} The spectrahedron $\G^+(q)$ is the intersection of the affine space $G + I_2$ with the cone $\Sym_{n+1}^+$, 
where $G$ is some Gram matrix of $q$ and $I_2$ is the linear space of symmetric matrices representing quadratic forms that vanish identically on $X$. 
By arguments similar to the proof of Lemma~\ref{lem:generaltop}, $q$ has a positive definite Gram matrix, which we can take to be $G$. 
Then the dimension of $\G^+(q)$ equals the dimension of $I_2$, which is one if and only if $X$ is a quadratic hypersurface. 
\end{proof}

We will now also give an interpretation of Gram spectrahedra in a toric setup. 
\begin{Rem}\label{rem:toric}
Let $P\subset\R^n$ be a lattice polytope whose vertices have nonnegative coordinates. We write $X_P$ for the Zariski closure of the image of the map
\[
m_P\colon (\C^\ast)^n \to \P^{N-1},  \ \ x\mapsto (x^{m_1}:x^{m_2}:\dots:x^{m_N}),
\]
where $N$ is the number of lattice points in $P$ and
$m_1,m_2,\dots,m_N$ are those lattice points in a fixed order. Note
that this does not quite agree with the usual notation in toric geometry (for
example in \cite{CLSMR2810322}) since we do not assume the polytope $P$
to be normal --- the map $m_P$ might not be an embedding of the torus into
$X_P$. If $P$ is normal, the projective variety $X_P$
defined above is isomorphic to the toric variety associated to
$P$. For our purposes, it is enough to assume that all lattice points
in $2P$ can be written as a sum of two lattice points in $P$. This property is called \emph{$2$-normal}. With
this assumption, we can interpret the vector spaces $\R[\x]_P$ and
$\R[\x]_{2P}$ in terms of the homogeneous coordinate ring of $X_P$.
Similarly to the Veronese embedding, polynomials $p\in\R[\x]_P$, whose
Newton polytopes are contained in $P$, are in 1-1 correspondence with
linear forms on $X_P$. Polynomials $f\in\R[\x]_{2P}$ can be represented by quadratic forms
restricted to $X_P$. In this case, a representing quadratic form is
not uniquely determined.  We will interpret the Gram spectrahedron of
a polynomial $f\in\R[\x]_{2P}$ in this setup.

The convex cone of symmetric matrices $A\in\Sym_N$ that represent a positive multiple of $f$, i.e.~$m_P^t A m_P = \lambda f$ for some $\lambda > 0$ (and the $0$ matrix), is the cone over the Gram spectrahedron $\GS(f)$ of $f$, which we denote by $\wh{\GS}(f)$. We can identify a matrix in $\wh{\GS}(f)$ with a nonnegative quadratic form on $\P^{N-1}$ that restricts to $f$ on $X_P$ (up to scaling), which means that it cuts out the divisor on $X_P$, which is uniquely determined by $f\in\R[\x]_{2P}$. The linear span $\wh{\G}(f)$ of the cone $\wh{\GS}(f)$ is the set of all quadratic forms that either cut out the divisor on $X_P$ defined by $f$ or vanish identically on $X_P$, i.e.~lie in the degree-$2$ part of the homogeneous ideal defining $X_P$. Note that we are homogenizing, that is, we are going from $\wh{\G}(f)$ to $\P(\wh{\G}(f))$. The hyperplane at infinity is the quadratic part of the vanishing ideal of $X_P$ and therefore independent of $f\in\R[\x]_{2P}$.

This point of view gives a natural interpretation of the boundary of the Gram spectrahedron.
If $f\in\R[\x]_{2P}$ is strictly positive, then the boundary of $\wh{\GS}(f)$ is the set of nonnegative extensions of $f$ to $\P^{N-1}$ that have a real zero in $\P^{N-1}$, which must necessarily be off the variety $X_P$.
\end{Rem}

The assumption that allows us to interpret polynomials in $\R[\x]_{2P}$ with quadratic forms on $X_P$ is satisfied for every lattice polytope $P$ such that every nonnegative polynomial in $\R[\x]_{2P}$ is a sum of squares. The main idea of the proof goes back to Motzkin's example of a nonnegative polynomial that is not a sum of squares and Reznick's generalizations in \cite{ReznickMR985241}.

\begin{Prop}[see also {\cite[Lemma~6.2]{BSVMR3486176}}]\label{Prop:normal}
Let $P\subset\R^n$ be a lattice polytope with vertices in $\Z_{\geq 0}^n$ and suppose that every nonnegative polynomial in $\R[\x]_{2P}$ is a sum of squares. Then every lattice point in $2P$ is a sum of two lattice points in $P$, i.e.~$P$ is $2$-normal.
\end{Prop}

\begin{proof}
Assume for contradiction that there is a lattice point $\beta\in 2P$ that cannot be written as a sum of two lattice points in $P$. Choose affinely independent vertices $\alpha_0,\dots,\alpha_{n}$ of $2P$ such that $\beta = \lambda_0 \alpha_0 + \dots + \lambda_n \alpha_n$ is a convex combination of these vertices. By the weighted arithmetic-geometric mean inequality, the polynomial $f = \sum \lambda_i x^{\alpha_i} - x^\beta$ is nonnegative. It cannot be a sum of squares because the coefficient of $x^\beta$ will be $0$ in any sum of squares of polynomials whose  Newton polytope is contained in $P$. This proves the claim.
\end{proof}

\section{Ranks on Gram Spectrahedra}\label{sec:Ranks}

As explained in Section~\ref{sec:Gram}, the length of a polynomial is
the minimum number of squares needed to represent it. In the context
of ring theory and quadratic forms, the maximal length of any sum of
squares is often called the \textit{Pythagoras number} (see
for example \cite[\S 4]{Sc09}). In terms of Gram spectrahedra, the
length of a polynomial is the minimum rank of any matrix in its Gram
spectrahedron. One can use both general rank-constraints for
spectrahedra and more specific methods for sums of squares to study
the possible ranks of Gram matrices.

First, we will look at the ranks of extreme points of spectrahedra in general.
\begin{Prop}[{\cite[Chapter 3]{PatakirangeMR1778223}}]\label{Prop:Patakirange}
Let $L\subset \Sym_N$ be an affine-linear space of dimension $m$. The rank $r$ of an extreme point of $L\cap \Sym_N^+$ satisfies
\[
\binom{r+1}{2} + m \leq \binom{N+1}{2}.
\]
Furthermore, if the affine-linear space $L$ is generic, the rank $r$ also satisfies
\[
m \geq \binom{N-r+1}{2}.
\]
\end{Prop}
These two inequalities define an interval of possible ranks for the
extreme points of a general spectrahedron, which is called the
\emph{Pataki interval}. The first inequality gives an upper bound on the rank $r$, whereas the second gives a lower bound.

\begin{proof}
We prove these two inequalities by a dimension count. Denote by $\V_r$ the variety of matrices of rank $\leq r$ in $\Sym_N$. 
Then $\V_r$ is ruled by linear spaces of dimension $\binom{r+1}{2}$, all of which are conjugate to the linear space of $r\times r$ symmetric matrices embedded into $\Sym_N$ as the upper left block, i.e. matrices of the form
\[
\begin{pmatrix}
A & 0 \\
0 & 0
\end{pmatrix}.
\]
The cone of positive semidefinite rank-$r$ matrices has non-empty interior in all of these linear spaces and a rank $r$ matrix is a relative interior point. If $L\cap \Sym_N^+$ has an extreme point $A$ of rank $r$, the affine-linear space $L$ must intersect the face $F_A$ of $\Sym_N^+$, in which $A$ is a relative interior point, in $A$ only. The span of $F_A$ is a linear space $U_A$ of 
dimension $\binom{r+1}{2}$ and the intersection $U_A\cap L$ is $0$-dimensional.  The first inequality then follows from the dimension formula in linear algebra.

For the second inequality, note that $\V_r$ has codimension $\binom{N-r+1}{2}$ in $\Sym_N$. A general linear space $L$ intersects $\V_r$ only if $\dim(L)$ is at least this codimension.
\end{proof}

Note that the \emph{smallest} rank of a positive semidefinite Gram matrix of a polynomial $f$ equals the smallest length of any sum of squares representation of $f$.
\begin{Prop}\label{Prop:smallestrank} \
\begin{enumerate}[(a)]
\item Let $r$ be the smallest rank of any matrix in the spectrahedron $L\cap \Sym_N^+$. Any matrix $A\in L\cap \Sym_N^+$ of rank $r$ is an extreme point of the spectrahedron.
\item For a general polynomial $f\in \R[\x]_{2P}$, the smallest rank $r$ of any positive semidefinite Gram matrix of $A$ satisfies the Pataki inequality
\[
m \geq \binom{N-r+1}{2},
\]
where $N$ is the number of lattice points in $P$ and $m$ is the dimension of the affine-linear space of Gram matrices of $f$.
\end{enumerate}
\end{Prop}

\begin{proof}
Part (a) follows from the fact that the rank of a convex combination $\lambda A + (1-\lambda)B$ of positive semidefinite matrices $A$ and $B$ is at least the minimum of the ranks of $A$ and $B$. Part (b) follows from a dimension count: Let $k$ be an integer that does not satisfy the Pataki inequality in the claim, then the codimension of the variety of symmetric matrices of rank at most $k$ is larger than $m$, the dimension of the kernel of the Gram map. So the image of this variety under the Gram map has codimension at least $1$. Therefore, a generic polynomial does not have a Gram matrix of rank at most $k$.
\end{proof}

\begin{Rem}
This Proposition \ref{Prop:smallestrank} shows that the rank of all extreme points of the Gram spectrahedron of a general polynomial lies in the Pataki interval for the ranks of extreme points of general spectrahedra.
\end{Rem}

Recall that the spanning polytopes $P$ for which every nonnegative polynomial in $\R[\x]_{2P}$ is a sum of squares are classified, see Theorem \ref{thm:NewtonPolytopes}. For varieties of minimal degree, Blekherman, Plaumann, Sinn, and Vinzant  \cite{BPSV2016arXiv160604387B} determined the lowest rank of a positive semidefinite Gram matrix of a general positive quadratic form. 
\begin{Thm}[{\cite[Theorem 1.1]{BPSV2016arXiv160604387B}}]\label{thm:BPSV}
Let $P\subset\R^n$ be an $n$-dimensional lattice polytope such that every nonnegative polynomial in $\R[\x]_{2P}$ is a sum of squares. The lowest rank of a positive semidefinite Gram matrix of a general nonnegative polynomial $f\in\R[\x]_{2P}$ is $n+1$, which is the smallest rank in the Pataki interval.
\end{Thm}

\begin{proof}
Since every nonnegative polynomial in $\R[\x]_{2P}$ is a sum of
squares, every lattice point in $2P$ is a sum of two lattice points in
$P$ by Proposition~\ref{Prop:normal}. So we can interpret every
polynomial whose Newton polytope is contained in $2P$ as a quadratic form on $X_P$. Therefore, the projective variety $X_P$ is an $n$-dimensional variety of minimal degree by \cite[Theorem~1.1]{BSVMR3486176}. So \cite[Theorem 1.1]{BPSV2016arXiv160604387B} implies that every nonnegative quadratic form on $X_P$ is a sum of $n+1$ squares. The fact that this is the smallest rank in the Pataki interval for this setup follows from the dimension count in \cite[Lemma 1.2]{BPSV2016arXiv160604387B}, which also implies that $n+1$ is the smallest rank of a positive semidefinite Gram matrix of a general polynomial.
\end{proof}

In most other cases, the smallest rank of an extreme point of a
general Gram spectrahedron is not known. Recently, Scheiderer showed
that every ternary sextic that is a sum of squares is a sum of $4$
squares, and every quartic in four variables that is a sum of squares is a
sum of $5$ squares \cite[Theorems 4.1 and 4.2]{ScheidererLengths}.

In both of these cases, the smallest rank of a positive semidefinite
Gram matrix is the smallest rank in the Pataki intervals. We now prove
a more general theorem that implies both of Scheiderer's cases and extends to arithmetically Cohen-Macaulay and linearly normal varieties of almost minimal degree, i.e.~such that $\deg(X) = \codim(X)+2$. The key invariant is the quadratic deficiency. A variety $X\subset\P^n$ with quadratic deficiency $\epsilon(X) = 1$ is either a hypersurface of degree at least $3$ (which is not interesting in our context) or a linearly normal variety of almost minimal degree, see Zak \cite[Proposition 5.10]{ZakMR1678545}.
\begin{Thm}\label{Thm:ACMalmost}
  Let $X\subset\P^n$ be an irreducible non-degenerate real projective
  variety with Zariski-dense real points. If $X$ is arithmetically
  Cohen-Macaulay and the quadratic deficiency
  $\epsilon(X)$ equals $1$, then
  every sum of squares in $\R[X]_2$ is a sum
  of $\dim(X)+2$ squares. This is the smallest rank in the Pataki
  interval.
\end{Thm}

\begin{proof}
Let $m$ be the dimension of $X$ and consider first the map 
\[
\phi_{m+1}\colon \left\{
\begin{array}[]{l}
\R[X]_1^{m+1} \to \R[X]_2 \\
(\ell_1,\ell_2,\dots,\ell_{m+1}) \mapsto \sum_{i=1}^{m+1} \ell_i^2 .
\end{array}\right.
\]
The differential of this map at $(\ell_1,\dots,\ell_{m+1})$ takes $m+1$ linear forms $(h_1,\dots,h_{m+1})$ to $2 \sum_{i=1}^{m+1} h_i \ell_i$. If $\ell_1,\dots,\ell_{m+1}$ do not have a common zero on $X$, the same dimension count as in \cite[Lemma 1.2]{BPSV2016arXiv160604387B}, using the fact that $X$ is arithmetically Cohen-Macaulay, shows that the image of $\dd \phi_{m+1}$ at $(\ell_1,\dots,\ell_{m+1})$ is a hyperplane in $\R[X]_2$. In fact, \cite[Proposition 3.5]{BSVMR3486176} shows that the algebraic boundary of $\Sigma_X$ has two irreducible components, the discriminant $D$ and the Zariski closure of the image of $\phi_{m+1}$. We now show that the image of the map
\[
\phi_{m+2} \colon \left\{
\begin{array}[]{l}
\R[X]_1^{m+2} \to \R[X]_2 \\
(\ell_1,\ell_2,\dots,\ell_{m+2})\mapsto \sum_{i=1}^{m+2} \ell_i^2
\end{array}
\right.
\]
is the cone of sums of squares in $\R[X]_2$. First, note that the differential of $\phi_{m+2}$ is surjective at $(\ell_1,\ell_2,\dots,\ell_{m+2})$, whenever the linear functionals $\ell_1,\dots,\ell_{m+2}$ are linearly independent and do not have a common zero on $X$. Indeed, the space of all quadratic forms $h_1\ell_1 + \dots + h_{m+1} \ell_{m+1}$, where $h_1,\dots,h_{m+1}$ vary over all linear forms, is a hyperplane in $\R[X]_2$ by the dimension count above; so adding $\ell_{m+2}^2$ will give all of $\R[X]_2$.
Let $S = \Sigma_X\setminus (\im(\phi_{m+1}) \cup D)$ be a subset of
the cone of sums of squares, where we remove the sums of $m+1$ squares
and the discriminant. Then $S$ is dense in $\Sigma_X$ and the subset
$\im(\phi_{m+2})\cap S$ is open and closed in $S$. Therefore,
$\im(\phi_{m+2})\cap S$ is a union of connected components of $S$. In
fact, it is all of $S$: The set $\Sigma_X\setminus D$ is connected
because $D\cap{\rm int}(\Sigma_X)$ has codimension $2$. So two connected components of $S$ are separated by the hypersurface $Z = \im(\phi_{m+1})$. In this case, there is a regular point $s$ of this hypersurface in the interior of $\Sigma_X$. In a small neighborhood of $s$, the complement of $Z(\R)$ in $\R[X]_2$ has two connected components. Let $H$ be the tangent hyperplane to $Z$ at $s$. Since it meets the interior of $\Sigma_X$, there are squares $s_1$ and $s_2$ in both half-spaces defined by $H$. So both connected components of the complement of $Z(\R)$ contain a sum of $m+2$ squares, namely $s + \epsilon s_1$ and $s + \epsilon s_2$ for sufficiently small $\epsilon>0$. In conclusion, the set $S$ is contained in the image of $\phi_{m+2}$. By the density of $S$ in $\Sigma_X$ and the fact that the map $\phi_{m+2}$ is closed, we conclude $\Sigma_X = \im(\phi_{m+2})$, proving the claim.
\end{proof}

From this, we obtain the results of Scheiderer mentioned above.
\begin{Cor}[{\cite[Theorem 4.1 and 4.2]{ScheidererLengths}}]
\item
\begin{enumerate}
\item Every ternary sextic that is a sum of squares is a sum of $4$ squares.
\item Every quartic sum of squares in four variables is a sum of $5$ squares.
\end{enumerate}
\end{Cor}

\begin{proof}
A ternary sextic corresponds to a quadratic form in $\R[\nu_3(\P^2)]$, the coordinate ring of the cubic Veronese embedding of $\P^2$, which is arithmetically Cohen-Macaulay. The quadratic deficiency of $\nu_3(\P^2)$ is $1$, so (1) follows from Theorem \ref{Thm:ACMalmost}. A quartic in $4$ variables corresponds to a quadratic form in $\R[\nu_2(\P^3)]$, the coordinate ring of the quadratic Veronese embedding of $\P^3$, which is again arithmetically Cohen-Macaulay with quadratic deficiency $1$.
\end{proof}

\begin{Rem}
  The two cases considered by Scheiderer in the above corollary
  correspond to toric varieties satisfying the conditions of
  Theorem~\ref{Thm:ACMalmost}. More generally, all such toric
  varieties have been classified. Namely, an embedded projective toric
  variety $X_P$ with normal lattice polytope $P$ is of almost minimal degree if and only if it is
  a Gorenstein del Pezzo variety (see Brodmann-Schenzel
  \cite[Theorem~6.2]{BS07}). In terms of the polytope, it means that the polytope $(n-1)P$ corresponding to the anti-canonical divisor up to translation is reflexive  (see \cite[Theorem 8.3.4]{CLSMR2810322}), i.e.~$P$ is Gorenstein of index $(n-1)$. These polytopes (without the normality assumption) have been
  classified by Batyrev and Juny in \cite{BJ10}. In the
  two-dimensional case, there are 16 reflexive polytopes, also
  discussed in the book of Cox, Little, and Schenck \cite[\S
  8.3]{CLSMR2810322}. Disregarding pyramids, which correspond to
  algebraic cones, there are only 37 Gorenstein polytopes of index $(n-1)$ in total, and their dimension is at most 5 by \cite{BJ10}. We have verified that all of these 37 lattice polytopes except for $P_1$ in the notation of \cite{BJ10} are normal, using the \textsc{Normaliz} software package \cite{Normaliz}.
\end{Rem}

Scheiderer also proves the following bound on the smallest rank of a positive semidefinite Gram matrix for ternary forms of degree at least $8$.
\begin{Thm}[{\cite[Theorem 3.6]{ScheidererLengths}}]\label{thm:scheiderer}
Every ternary form of degree $2d$ that is a sum of squares is a sum of $r$ squares, where $r$ is either $d+1$ or $d+2$.
\end{Thm}

\begin{Rem}
For $d\geq 4$, even $d+1$ is not the lower bound in the Pataki interval because $r=d$ satisfies the Pataki inequality. So the ranks of extreme points of a general Gram spectrahedron in the case of ternary forms of large degree satisfy the Pataki inequalities but the ranks do not achieve the lower bound of the interval.
\end{Rem}

In analogy to real Waring rank, we discuss typical (sum of squares) lengths.
\begin{Def}
A length $r$ is \emph{typical} in $\R[\x]_{2P}$ if the set of sums of squares of length $r$ has non-empty interior in $\R[\x]_{2P}$.
\end{Def}

\begin{Prop} Let $\Sigma(r)\subset\R[\x]_{2P}$ be the set of
  polynomials of length $\leq r$.
  \begin{compactenum}[(a)]\item The length is a lower semi-continuous
    function, i.e.~$\Sigma(r)$ is closed.
   \item If the interior of $\Sigma(r)$ is nonempty, then it is dense
     in $\Sigma(r)$.
  \end{compactenum}
\end{Prop}

\begin{proof}
  Let $\phi\colon\R[\x]_P^r\to\R[\x]_{2P}$ be the map
  $(p_1,\dots,p_r)\mapsto \sum_{i=1}^rp_i^2$. The map $\phi$ is
  homogeneous and $\phi(p_1,\hdots,p_r)\neq 0$ whenever
  $(p_1,\hdots,p_r)\neq 0$. So we can view $\phi$ as a continuous map from
  $\P(\R[\x]_P^r)$ to $\P(\R[\x]_{2P})$, taking the Euclidean
  topology on both projective spaces. As a continuous map between
  compact Hausdorff spaces, it is both proper and closed. In
  particular, the image of $\phi$ in $\R[\x]_{2P}$ is
  closed. This proves (a). For (b), it suffices to note that if ${\rm
    int}\bigl(\Sigma(r)\bigr)\neq\emptyset$, the
  differential of $\phi$ is generically surjective. Thus $\Sigma(r)$
  is locally of full dimension at every point.
\end{proof}

\begin{Cor}
All lengths between the minimal typical length and the maximal length are typical. Moreover, the maximal typical length is the maximal length.
\end{Cor}

This statement is also proved by Scheiderer \cite[Corollary~2.5]{ScheidererLengths}.

\begin{proof}
Let $r$ be a typical length and suppose that every sum of $r+1$ squares can be written as a sum of $r$ squares. We show that $r$ is the maximal typical length. Indeed, every sum of $r+k$ squares can be inductively shortened to a sum of $r$ squares. This argument shows that the set of polynomials of length $k$ is strictly contained in the set of polynomials of length $k+1$ for all $k$ that are smaller than the maximal typical length. Now the claim follows from the preceding proposition.
\end{proof}

The following statement follows directly from Proposition~\ref{Prop:smallestrank}.
\begin{Prop}
  Let $P\subset\R^n$ be an $n$-dimensional lattice
  polytope. The minimal typical length in $\R[\x]_{2P}$ satisfies the
  Pataki inequalities from Proposition~\ref{Prop:Patakirange}.
\end{Prop}

In general, it is difficult to determine the typical lengths for a
given polytope. Even the case of simplices, corresponding to all forms
of fixed degree, is open. Under the assumption that every nonnegative polynomial is a sum of squares, we can determine the minimal typical and maximal length, which happen to be equal.

\begin{Rem}
If the projective variety $X_P$ is of minimal degree, there is only one typical length $n+1$ by \cite[Theorem~1.1]{BPSV2016arXiv160604387B}. If $P$ is normal and the toric variety $X_P$ is of almost minimal degree (it is automatically arithmetically Cohen-Macaulay because it is toric \cite[Theorem~9.2.9]{CLSMR2810322}), there is also only one typical length, namely $n+2$ (Theorem~\ref{Thm:ACMalmost}) because $n+1$ is not typical, see \cite[Proposition~3.5]{BSVMR3486176}. In general, there can be several typical lengths. The first example is ternary forms of degree $8$. It can be verified by symbolic computation that length $4$ is typical. By Scheiderer's result \cite[Theorem~3.6]{ScheidererLengths} (see Theorem~\ref{thm:scheiderer}), we know that the maximal length in this case is either $5$ or $6$. In any case, length $5$ is also typical.
\end{Rem}

Coming back to the relation between sums of squares and general
convexity results, we have seen several ways in which the convex
geometry of Gram spectrahedra is similar to that of generic
spectrahedra of the same matrix size and dimension, for example, in
their Pataki intervals.  There are, however, some ways in which Gram
spectrahedra behave non-generically.  As discussed in the next
section, the algebraic degree of semidefinite programming over Gram
spectrahedra seems to be lower than the generic degree.  One reason
for this special behavior is that the affine Gram spaces $\G(f)$ for
$f\in \R[\x]_{2P}$ all share the same geometry at infinity.

\begin{Rem}
  Let $P\subset \R^n$ be a lattice polytope.  The affine-linear spaces
  $\G(f)$, where $f\in \R[\x]_{2P}$, are all translates of the same
  linear space $\G(0) = \{A\in \Sym_N : \m_P^tA\m_P = 0\}$. If the
  restriction of the determinant to $\G(0)$ is non-zero, then it is
  the leading form of the restriction $\det(\G(f))$, for any
  $f\in \R[\x]_{2P}$.  This implies that if $\det(\G(0))$ is non-zero
  and irreducible, then the determinant $\det(\G(f))$ is irreducible
  for every $f\in \R[\x]_{2P}$.  Based on computational evidence, this
  is often, but not always, the case.  For example, for binary forms
  corresponding to $P = d\Delta_1$, the restriction of the determinant
  to $\G(0)$ is identically zero for $d=1$, non-zero and reducible for
  $2\leq d \leq 4$, and non-zero and irreducible for $d\geq 5$.
\end{Rem}

\section{Examples: Binary forms and Ternary Quartics}\label{sec:Examples}
\subsection{Binary Forms}\label{subsec:Binary}
Binary forms are homogeneous polynomials in two variables. After dehomogenizing, these correspond to univariate polynomials. 
A nonnegative binary form can always be written as a sum of two squares. Following \cite[Example 2.13]{CLR}, 
we can compute the number of such representations. 

\begin{Prop}\label{prop:binaryforms}
Let $f\in \R[s,t]_{2d}$ be a positive binary form with distinct roots.
Over $\C$, the affine space $\G(f)$ contains $\frac{1}{2}\binom{2d}{d}$ matrices of rank two. The Gram spectrahedron $\GS(f)$ contains $2^{d-1}$ matrices of rank $2$. If $d$ is odd, then these are the only real rank two matrices. If $d$ is even, there are an additional $\frac{1}{2}\binom{d}{d/2}$ real indefinite matrices of rank $2$ in $\G(f)$.
\end{Prop}

\begin{proof}
Over $\C$, writing $f$ as a sum of two squares $f =p^2+q^2$ is the same as giving a factorization $f = (p+i q)(p-i q)$. 
Therefore, a sum-of-two-squares representation yields a partition of the roots of $f$ into two subsets of size $d$.
Moreover, two equivalent representations as a sum of two squares yield the same partition. 
Indeed, any representation of $f$ as a sum of two squares, which is equivalent to $f = p^2 + q^2 = (p+i q)(p-i q)$, has the form
\[
 f \ \ = \ \ (\cos(\theta) p + \sin(\theta) q)^2 + (\sin(\theta) p - \cos(\theta) q)^2 \ \ = \ \ \bigl(e^{i\theta}( p -iq)\bigl)\cdot \bigl(e^{-i\theta}( p +iq )\bigl).
\]
This equivalent representation gives the same partition of the roots of $f$ into the roots of $p+iq$ and $p-iq$. The number of such partitions is $\frac{1}{2}\binom{2d}{d}$.

For such a partition to be real, complex conjugation must either fix
both blocks of $d$ roots or swap them. If both blocks are fixed, $d$
must be even, and the number of such partitions is
$\frac{1}{2}\binom{d}{d/2}$. This gives
$f = g h = \frac{1}{2}((g+h)^2 - (g-h)^2)$, where both $g$ and $h$ are
positive. If the two blocks of the partition are swapped by conjugation,
then each conjugate pair has one element in each block. The number of
such partitions is $2^{d-1}$. These give an expression
$f = g \overline{g} = ({\rm Re}(g))^2+ ({\rm Im}(g))^2$ as a sum of
two squares of real polynomials.
\end{proof}

Ranks of the matrices in the convex hull of the $2^{d-1}$ matrices of
rank two are discussed at the end of Section~\ref{sec:Hermitian}.
Extreme points on the Gram spectrahedron of rank greater than $2$ are
not well understood.  For example, in high dimensions, no instance of one single
spectrahedron with extreme points of all ranks in the Pataki interval
is known.  The Gram spectrahedron of a binary form seems like a good
candidate to have this property.

\begin{Ques}
  Does there exist a binary form whose Gram spectrahedron has 
  extreme points of all ranks in the Pataki interval?
\end{Ques}

We will discuss the Gram spectrahedra of binary forms of low degree. Let $f\in\R[s,t]_{2d}$ be a binary form of degree $2d$.
If $2d=2$, the Gram spectrahedron is $0$-dimensional. If $f = a_2s^2 + a_1st + a_0t^2$ is nonnegative, then its Gram spectrahedron consists of the positive semidefinite matrix
$
{\small \begin{pmatrix}
a_2 \! & \! \frac{1}{2}a_1 \\
 \frac{1}{2}a_1 \! & \! a_0
\end{pmatrix}}$.
The determinant of this matrix is the discriminant of $f$. 
If $f$ has a zero, then $f = \ell^2$ for the linear form vanishing at the root of $f$ and its Gram matrix has rank $1$. If $f$ is strictly positive, then $f$ is the sum of two squares and its Gram matrix has rank $2$. This representation can be found by completing the square, and it is the unique way of writing $f$ as a sum of squares up to orthogonal changes of coordinates.

If $f$ is a positive binary form of degree $2d = 4$, then $\GS(f)$ is a line segment, which is the convex hull of the two decompositions of $f$ as the sum of two squares, say $f = p_1^2 + q_1^2$ and $f = p_2^2 + q_2^2$, cf.~Proposition~\ref{prop:quadratichypersurf}. The representation 
\[
f  = \frac12 (p_1^2 + q_1^2) + \frac12 (p_2^2 + q_2^2)
\]
is equivalents to  a sum of three squares, because the rank of a Gram matrix of $f$ in the relative interior of the Gram spectrahedron is $3$.
\medskip

\subsection{Binary sextics and Kummer surfaces}
The case of binary sextics has interesting connections to classical algebraic geometry. We fix the notation
\[
f \ \ = \ \ a_6s^6 -6a_5s^5t +15a_4s^4t^2 -20a_3s^3t^3 +15a_2s^2t^4 -6a_1st^5 + a_0t^6
\]
for the coefficients of $f$. Then $\GS(f)$ is the matrices in $\Sym_4^+$ of the form
\begin{equation}\label{eq:GramSextic}
\left(
\begin{array}[]{cccc}
a_6  & -3 a_5  & 3 a_4 + z & -a_3 -y \\
-3 a_5  & 9 a_4 -2 z & -9 a_3 +y & 3 a_2 +x \\
3 a_4 +z & -9 a_3 +y & 9 a_2 -2 x & -3 a_1  \\
-a_3 -y & 3 a_2 +x & -3 a_1  & a_0 
\end{array}
\right).
\end{equation}
We homogenize the linear forms in the entries of this matrix with
respect to a homogenizing variable $w$. Generically, the determinant
of this matrix defines a quartic surface in $\P^3$, a Kummer surface;
see \cite[Section 41]{CobleMR733252}. It has $16$ singular points,
which are nodes. Six of these nodes have homogeneous coordinates of
the form $(u_i^2 : u_i : 1 : 0)$, where $u_1,\dots,u_6$ are the distinct complex roots of $f(s,1)$. These nodes correspond to matrices of rank $3$. The remaining ten nodes give rank-$2$ matrices. These correspond to the $10$ distinct representations of $f$ as a sum of two squares over $\mathbb{C}$, see Proposition~\ref{prop:binaryforms} and \cite[Section 5]{QuarticSpectrahedra}. 

Kummer surfaces have several different descriptions and nice geometric properties, see \cite[Chapter 6]{GH}. For example, the dual projective variety of a Kummer surface is again a Kummer surface. With our choice of coordinates, its equation is \\[-2em]
\par\nobreak
{\small
\begin{align*}
 &F = W^2\bigl(-Y^2+4XZ\bigl) + 2W\bigl(a_0X^3+3a_1X^2Y+3a_2XY^2+a_3Y^3+3a_2X^2Z+6a_3XYZ\\
 &+3a_4Y^2Z+3a_4XZ^2+3a_5YZ^2+a_6Z^3\bigl)+ \bigl(9(a_0a_2-a_1^2)X^4 + 18(a_0a_3-a_1a_2)X^3Y\\
 &+3(5a_0a_4-2a_1a_3-3a_2^2)X^2Y^2+6(a_0a_5-a_2a_3)XY^3 + (a_0a_6-a_3^2)Y^4\\ 
 &+6(-a_0a_4+10a_1a_3-9a_2^2)X^3Z+ 6(-a_0a_5+12a_1a_4-11a_2a_3)X^2YZ  \\
&+  2(-a_0a_6+18a_1a_5-9a_2a_4-8a_3^2)XY^2Z + 6(a_1a_6-a_3a_4)Y^3Z + 9(a_4a_6-a_5^2)Z^4 \\
&+ (a_0a_6-18a_1a_5+117a_2a_4-100a_3^2)X^2Z^2 + 6(-a_1a_6+12a_2a_5-11a_3a_4)XYZ^2 \\
&+  3(5a_2a_6-2a_3a_5-3a_4^2)Y^2Z^2+ 6(-a_2a_6+10a_3a_5-9a_4^2)XZ^3 + 18(a_3a_6-a_4a_5)YZ^3\bigl),
\end{align*}}
\noindent which has the form $F = W^2F_2  + 2W  F_1 + F_0$, where $F_2,F_1, F_0$ are polynomials in $X,Y,Z$. 
This equation was computed symbolically by exploiting the $16_6$ configuration of the Kummer surfaces as follows. Each node of the dual Kummer surface defines a plane in the primal projective space, and each of these $16$ planes contains six nodes of the primal Kummer surface. The plane $\{w=0\}$ contains the six nodes $(u_i^2 : u_i : 1 : 0)$ given by the zeros of $f$. We find the remaining $15$ nodes of the dual Kummer surface by taking every combination of $6$ nodes that lie on a plane and computing its defining equation. Given the sixteen nodes of the dual Kummer surface, we find its equation by solving the linear system of equations in the coefficients imposing the condition that the surface is singular at the $16$ nodes. We find this equation first in terms of the roots $u_1,\ldots,u_6$ of $f$ and then express it in terms of $a_0,\ldots,a_6$, which is possible because the expression is symmetric in the $u_i$. 

Both the primal and dual surfaces are real Kummer surfaces with four real nodes and one can write an explicit 
real linear transformation taking one to the other.  
The dual Kummer surface has a node at the point $[0:0:0:1]$, corresponding to the hyperplane at 
infinity in the primal projective space.  Any real linear transformation $\check{\P}^3 \rightarrow \P^3$ taking the dual Kummer surface 
to the primal must map this node to one of the four real rank-2 Gram matrices of $f$. 
As discussed in Section~\ref{subsec:Binary}, the four real rank-2 Gram matrices of $f$ 
correspond to factorizations of $f$ as a Hermitian square $f = g\overline{g}$ and thus to 
partitions of the roots $u_1,u_2,u_3,u_4,u_5,u_6$ into two conjugate sets of three. 
The linear transformation from the dual to primal Kummer surface can be defined over the field 
extension of this split. For example, a transformation can be written symbolically in terms of the elementary symmetric polynomials of $u_1, u_2, u_3$ and $u_4, u_5, u_6$.
If the roots $\{u_1, u_2, u_3\}$ are conjugate to $\{u_4, u_5, u_6\}$, then this transformation is real and takes the node $[0:0:0:1]$ 
to the rank-2 Gram matrix of $f$ corresponding to the representation of $f$ as $g\overline{g}$ where $g = (t-u_1)(t-u_2)(t - u_3)$. 
A symbolic expression for such a linear transformation in these coordinates is available online
 (\url{http://www4.ncsu.edu/~clvinzan/misc/KummerTransformation.m2}). 

One intriguing consequence of the existence of this real linear
transformation is that the dual Kummer surface $V(F)$ also bounds a spectrahedron
(namely the image of $\GS(f)$ under the inverse of this map).  
However, this spectrahedron is \emph{not} the convex body dual to $\GS(f)$. 
The real points of the dual Kummer surface $V_{\R}(F)\subset \check{\P}^3(\R)$ can be written as the union of two semialgebraic subsets, 
each homeomorphic to a sphere and intersecting in the four real nodes. The inner sphere bounds 
a spectrahedron.  The outer sphere overlaps with the Euclidean boundary of the convex dual of $\GS(f)$. 
Specifically, this overlap consists of points in $\check{\P}^3$ whose corresponding hyperplane in $\P^3$ is tangent to $\GS(f)$ at a matrix of rank 3.

With the equation for $F$  in hand, we can determine the algebraic degree of semidefinite programming introduced in \cite{algDegNRS} 
over the Gram spectrahedron of a binary sextic and symbolically find the optimal value function, see \cite[Section~5.3]{RoStMR3050244}.
\begin{Prop}
The algebraic degree of semidefinite programming over the Gram spectrahedron of a general binary sextic is $(10,2)$ in ranks $(2,3)$.
\end{Prop}
This means that the optimality conditions for a general linear form over the Kummer surface will give two critical points of rank $3$. 
A general linear form in the affine-linear space $\G(f)$ gives the last three coordinates $X,Y,Z$ on the dual projective space. 
The critical points are given by the two solutions in $W$ of the equation of the dual Kummer surface for these fixed values of $X,Y,Z$. 
Specifically, given a cost vector $c= (c_1,c_2,c_3)\in \Q^3$, the rank-3 critical points $(x,y,z)$ satisfy  
\[ c_1x+c_2y+c_3z  \ \ = \ \ \frac{- 2F_1(c)  \pm \sqrt{F_1(c)^2 - 4\cdot F_0(c)\cdot F_2(c) }} {2 F_2(c)}.\]
One can use the KKT equations to solve for the two corresponding rank-$3$ matrices, using linear algebra,
and thus write down these matrices over a quadratic field extension of $\Q$. 
The optimal point on the Gram spectrahedron is either one of these two rank-$3$ matrices or one of the ten 
matrices of rank $2$, which are all critical points of this optimization problem.

\begin{figure}
\begin{center} 
\includegraphics[height=1.2in]{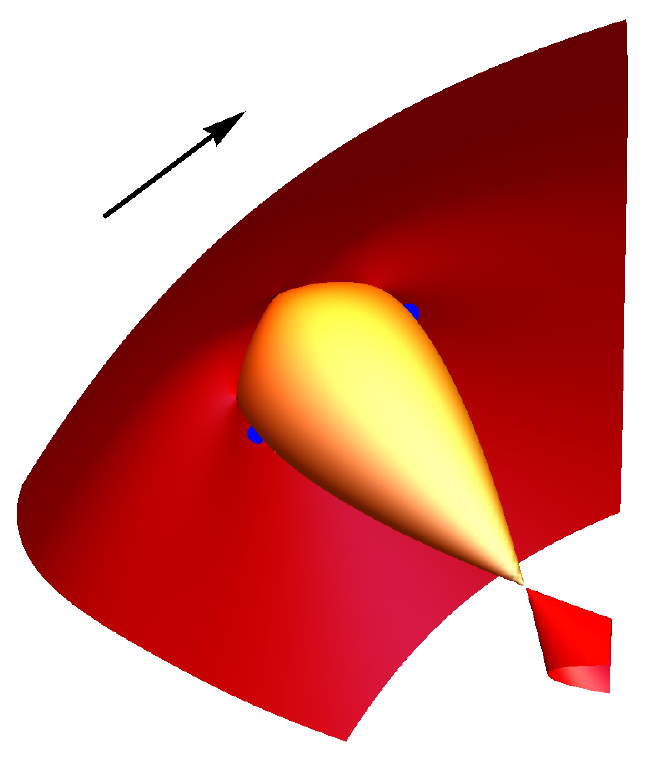} \ \ \ \ \
\includegraphics[height=1.2in]{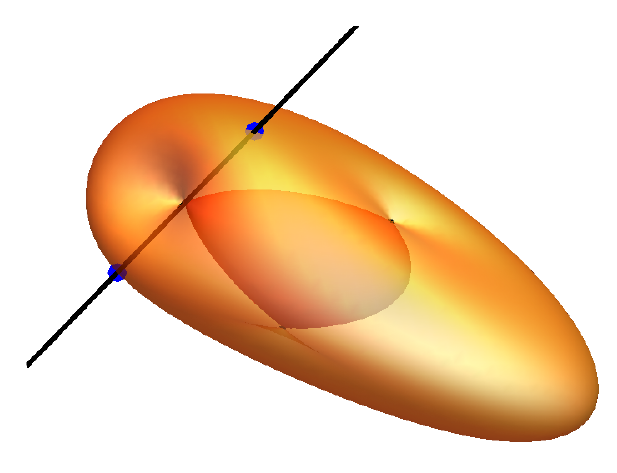} \ \ \ \ \
\includegraphics[height=1.2in]{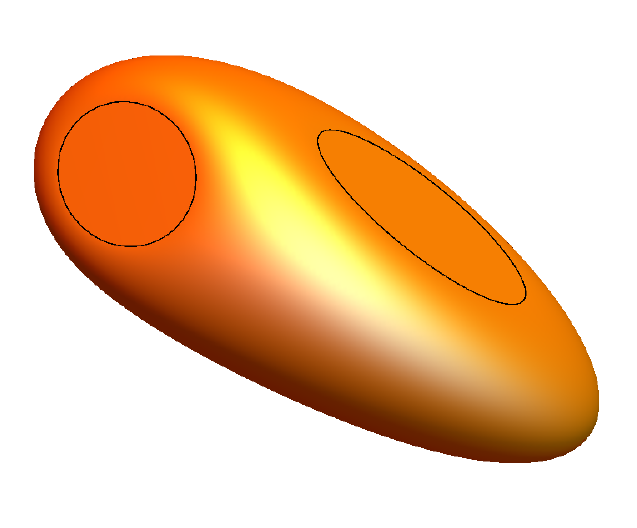}
\end{center}
\caption{\label{fig:Gram} 
The Kummer surface bounding $\GS(f)$, dual Kummer surface, and dual convex body of $\GS(f)$ from Example~\ref{ex:BinarySextic}. 
}
\end{figure}

\begin{Exm}\label{ex:BinarySextic}
Consider the binary sextic 
\[
f = s^6 -2s^5t +5s^4t^2 -4s^3t^3 +5s^2t^4 -2st^5 + t^6
\]
and the linear form $x-z$ on the spectrahedron $\GS(f)$, shown in Figure~\ref{fig:Gram}. This gives the coordinates $c =(X,Y,Z) = (1,0,-1)$ on the dual projective space.
The critical points are given by the two solutions in $W$ of the equation of the dual Kummer surface, which are $W= \pm \sqrt{5}$. 
We can solve for the corresponding points in the Kummer surface over $\Q[\sqrt{5}]$ 
and find $(x,y,z) = -\bigl((1\pm\sqrt{5} )/2, 1/5, (1 \mp \sqrt{5} )/2\bigr)$, corresponding to the rank-$3$ Gram matrices
\[
{\begin{pmatrix}
1& -1& \frac{1+\sqrt{5}}{2}& 0\\
 -1& 4-\sqrt{5}& -2& \frac{1-\sqrt{5}}{2}\\
\frac{1+\sqrt{5}}{2}& -2& 4+\sqrt{5}& -1\\
  0& \frac{1-\sqrt{5}}{2}& -1& 1
\end{pmatrix}}  
\ \ \ \text{ and } \ \ \ 
{\begin{pmatrix}
1& -1& \frac{1-\sqrt{5}}{2}& 0\\
 -1& 4+\sqrt{5}& -2& \frac{1+\sqrt{5}}{2}\\
\frac{1-\sqrt{5}}{2}& -2& 4-\sqrt{5}& -1\\
 0& \frac{1+\sqrt{5}}{2}& -1& 1
\end{pmatrix}.}
\]
In this case, the rank-$3$ matrices corresponding to these critical points 
are both positive semidefinite and actually maximize and minimize
$x-z$ over $\GS(f)$. Each of the four real rank-2 matrices in $\GS(f)$ gives a hyperplane forming part of the 
algebraic boundary of the dual convex body. 

The dual Kummer surface $V(F)\subset \check{\P}^3$ again bounds a spectrahedron that 
lies strictly inside the convex body dual to $\GS(f)$.  We can see this 
by performing a real change of coordinates that takes the original Kummer surface to $V(F)$. 
In this example, $F$ equals the determinant of the matrix \eqref{eq:GramSextic}, up to a scalar multiple,
on the affine chart $X - Y/5 - Z +W= 1$ taking 
\[
x = X+Y+Z+1, \ \  y = X+Y-Z-\frac{1}{5}, \ \ z = X-Y-3 Z-1.
\]
\end{Exm}

\subsection{Ternary Quartics}
The case of ternary quartics ($n=3$, $2d=4$) has an interesting
history, starting from Hilbert's famous theorem
\cite{HilbertMR1510517} that every nonnegative ternary quartic is a sum of three squares. In
\cite{PRMR1803369}, Powers and Reznick determined, via computational experiments, that 
a generic positive ternary quartic has $63$ complex, $15$ real, and $8$ real positive semidefinite 
Gram matrices of rank $3$. They later proved these counts, together
with Scheiderer and Sottile, in \cite{MR2103198}.

The Gram spectrahedron $\GS(f)$ of a positive ternary quartic
$f\in\R[x,y,z]_4$ is a six-dimensional spectrahedron in the
$21$-dimensional space $\Sym_6$. Thus any boundary
point is a Gram matrix of rank at most $5$, so that $f$ is a sum of at most $5$
squares. In fact, an intriguing result of Barvinok in \cite{Bar01} about
compact spectrahedra of a particular codimension can be applied here
to show that $\GS(f)$ contains a Gram matrix of rank at most
$4$. However, the number $3$ coming from Hilbert's theorem has not
been obtained using methods from general convexity alone.

Gram spectrahedra of ternary quartics were also studied by Plaumann,
Sturmfels, and Vinzant in \cite[\S 6]{PSV}. It is shown there that, for
generic $f$, the eight \textit{vertices} of $\GS(f)$, i.e.~the
positive semidefinite Gram matrices of rank $3$, can be divided into
two groups of four such that the line segment between two vertices is
contained in the boundary of $\GS(f)$ if and only if the two vertices
belong to the same group. These edges consist of Gram matrices of
rank $5$. This involves a detailed analysis of the combinatorial
structure of the bitangent lines of the plane quartic curve defined by
$f$, which can also be used to identify the vertices of $\GS(f)$.

The algebraic degree of semidefinite programming over the Gram
spectrahedron of a general ternary quartic was also studied
experimentally in \cite{PSV}. The authors observed that the
algebraic degree of the rank-5-locus seems to be $1$ (while it is $32$ for
general spectrahedra of dimension $6$ in $\Sym_6(\R)$). In particular,
this implies that a general rational ternary quartic possesses a
rational Gram matrix
of rank $5$. We believe that the geometric explanation for this fact
should be similar to what we observed in the case of Kummer surfaces
above, but it seems harder to derive an explicit formula analogous to
the equation of the dual Kummer surface that we found in terms of the coefficients of the binary sextic.

\begin{Ques} What is a formula for the hypersurface dual to the boundary of the Gram spectrahedron of a general 
ternary quartic? 
\end{Ques}

\section{Hermitian Gram spectrahedra}\label{sec:Hermitian}

Another way to certify nonnegativity of a polynomial
$f\in \R[\x]_{2P}$ is to write it as a \emph{Hermitian sum of
  squares}, $f=p_1\overline{p_1}+\cdots+p_r\overline{p_r}$,
where $p_1, \hdots, p_r \in \C[\x]$.  Overline denotes the action of 
complex conjugation on the coefficients of $p$,
i.e.~$\ol{\sum_\alpha c_{\alpha} x^{\alpha} } = \sum_\alpha
\ol{c_{\alpha}} x^{\alpha}$.
A Hermitian square $p\cdot \ol{p}$ is a sum of two squares
$\Re(p)^2 + \Im(p)^2$ of real polynomials $\Re(p) = (p+\ol{p})/2$ and
$\Im(p) = (p-\ol{p})/2i$. Thus a real polynomial is a real sum of
squares if and only if it is a Hermitian sum of squares. However,
Hermitian sums of squares enjoy some technical advantages for certain questions.  
As in the real case, the existence of
a Hermitian sum-of-squares representation is equivalent to the
feasibility of a (Hermitian) semidefinite program. 
Let $\Herm_N$ denote the space of complex Hermitian $N\times N$ matrices and 
$\Herm_N^+$ the cone of positive semidefinite matrices in $\Herm_N$. 

\begin{Prop}\label{Prop:HermGramMethod}
A polynomial $f\in\R[\x]_{2P}$ is a sum of $r$ Hermitian squares if
and only if there is a matrix $A\in\Herm_N^+$ of rank at most $r$ such that
\[
f = \m_P^t A \m_P.
\]
\end{Prop}

The proof is analogous to that of Proposition~\ref{Prop:GramMethod}.
%
We can use positive semidefinite Hermitian Gram matrices to represent equivalence classes of representations as Hermitian sums of squares, modulo the action of the group of unitary matrices. 

\begin{Def}\label{Def:HermGramSpec} If $\dim(\R[\x]_P) = N$, then for $f\in\R[\x]_{2P}$, we define
\[ \h(f) \ = \  \{A \in \Herm_N\ \colon \  \m_P^t A \m_P = f \} \ \ \text{ and } \ \ \HS(f) \ = \ \h(f) \cap \Herm_N^+  \]
to be the space of \emph{Hermitian Gram matrices} and the \emph{Hermitian Gram spectrahedron} of $f$, 
respectively.  The \emph{Hermitian (sum-of-squares) length} of $f\in\R[\x]_{2P}$ is the minimum rank in $\HS(f)$, 
i.e.~the smallest $r$ for which $f = p_1\overline{p_1}+\cdots+p_r\overline{p_r}$.
\end{Def}
%

\begin{Prop}\label{prop:HermMinRank}
If the real length of $f\in \R[\x]_{2P}$ is $r$, then its Hermitian length is $\lceil r/2 \rceil$. 
If the Hermitian length of $f$ is $k$, its
real length  is $2k-1$ or $2k$.
\end{Prop}

\begin{proof}
First, let $r$ be the real length of $f$. Fix a positive semidefinite Hermitian Gram matrix $Q \in \HS(f)$ of $f$ 
with minimum rank. Then $(Q+\ol{Q})/2$ belongs to $\GS(f)$ and has rank at most $2\rank(Q)$. 
Therefore $r \leq 2\rank(Q)$, which gives that $\lceil r/2 \rceil \leq \rank(Q)= \min\{\rank( \HS(f))\}$. 
For the reverse inequality, take a representation $f = p_1^2 + \hdots + p_r^2$ as a sum of $r$ real squares. We can write $f$ as
\[f \  = \   \sum_{k=1}^{\lfloor r/2\rfloor }(p_{2k-1} + i p_{2k})(p_{2k-1} - i p_{2k}) \ + \ \delta_{\text{odd}} \cdot p_r^2,\]
where $\delta_{\text{odd}}$ is 1 if $r$ is odd and $0$ otherwise. This gives a representation of $f$ as a sum of $\lceil r/2 \rceil$ 
Hermitian squares. Hence $\lceil r/2 \rceil \geq  \min\{\rank( \HS(f))\}$. 

If $k$ is the Hermitian length of $f$, then
$f$ is a sum of $2k$ real squares but not of $2(k-1)$ real
squares. So the real length of $f$ is either $2k$ or
$2k-1$.
\end{proof}

Combined with Theorem~\ref{thm:BPSV}, this gives the following: 

\begin{Cor} \label{cor:HermLength}
Let $P\subset\R^n$ be an $n$-dimensional lattice polytope with
vertices in $\Z_{\ge 0}^n$ and suppose that every 
nonnegative polynomial with Newton polytope $2P$ is a sum of squares. The lowest rank of a positive semidefinite 
Hermitian Gram matrix of a general nonnegative polynomial $f\in\R[\x]_{2P}$ is $\lceil (n+1)/2\rceil$.\qed
\end{Cor}

Interestingly, this is not always the minimum rank in the \emph{Hermitian Pataki interval}.
Arguments analogous to the proof of Proposition~\ref{Prop:Patakirange} show that  
for a generic affine-linear space $L \subset \Herm_N$ 
of codimension $c$, the rank $r$ of an extreme point of $L\cap \Herm^+_N$ satisfies
\[
(N-r)^2 + c \ \leq \ N^2  \ \ \text{ and } \ \ r^2 \ \leq \ c.
\]
If $P\subset\R^n$ is one of the polytopes in
Theorem~\ref{thm:NewtonPolytopes} with $N$ lattice points, then the
dimension of $\R[\x]_{2P}$ is $(n+1)N - \binom{n+1}{2}$. For
$f\in \R[\x]_{2P}$, this is the codimension of the affine-linear space
$\h(f)$ in $\Herm_N$.  After simplifying, we then
find that the minimum rank in the Hermitian Pataki interval is the
minimum $r$ such that
\[(1+n-2r)N + r^2 - 2r +n(n+1)/2 \ \leq \  0.\]
The rank $\lceil (n+1)/2\rceil$ satisfies this inequality. Furthermore,  
for fixed $n$ and sufficiently large $N$,  the term $(1+n-2r)N$ dominates the expression above, 
in which case $\lceil (n+1)/2\rceil$ is the minimum rank in the Hermitian Pataki interval.

\begin{Exm}Let $P$ be the polytope $\Delta_5 \times [0,1]$, i.e. 
the Cayley polytope of Theorem~\ref{thm:NewtonPolytopes}(2) with $d_1=\hdots =d_6=1$. 
This is a $6$-dimensional polytope with $N = 12$ lattice points. 
For  general $f\in \R[\x]_{2P}$, the space of Gram matrices $\h(f)$ 
has codimension $63$ in  $\Herm_{12}$. The Pataki interval is $\{3,4,5,6,7\}$. 
However, by Corollary~\ref{cor:HermLength}, 
the minimum rank of a matrix in $\HS(f)$ is $\lceil (n+1)/2 \rceil = 4$. 

After increasing $N$ to $13$, the minimum rank in the Hermitian Pataki
interval is indeed $4$.  For example, let $P$ be the Cayley polytope
with $d_1 = 2$ and $d_2=\hdots =d_6=1$.  Given $f \in \R[\x]_{2P}$,
the affine-linear space $\h(f)$ has codimension $70$ in $\Herm_{13}$.
The Hermitian Pataki interval for $N=13$, $c=70$ is $\{4,5,6,7,8\}$.
\end{Exm}

One benefit of the Hermitian perspective is that the kernels of
Hermitian matrices are complex linear spaces, rather than real.  For
the results below, it will be particularly useful to consider kernels
of Hermitian Gram matrices coming from complex points, i.e.~vectors of
the form $\m_P(x)$ for $x\in \C^n$.  We can use this to understand
some faces of Hermitian Gram spectrahedra. 

\begin{Prop}\label{prop:HermFactor}
If $f\in \R[\x]_{2P}$ factors as $g\overline{g} \cdot h$, where $g\in \C[\x]$ and $h\in \R[\x]$, 
then $\HS(h)$ is linearly isomorphic to a face of $\HS(f)$. 
If $g$ is square-free, this face consists of matrices whose kernel
contains the vectors $m_P(x)$ for $x\in\V_\C(g)$.
\end{Prop}

\begin{proof}
By induction, it suffices to consider a square-free factor $g$. 
Without loss of generality, we can also assume that $2P$ equals the Newton polytope of $f$. 
Then $2P$ is the Minkowski sum of the polytopes $2\new(g)$ and $\new(h)$. 
Therefore, we can write $\new(h)$ as $2Q$ for some $Q\subset \R^n$ with integer vertices.
We see that $P$ is the Minkowski sum $\new(g)+Q$. 

The polynomials $\{x^\alpha \cdot g : \alpha \in Q\}$ are contained in $\C[\x]_P$.  
Thus there is a complex matrix $U$ such that $\m_P^t U = g \cdot \m_Q^t$. Let $L$
be the linear space of matrices in $\Herm_N$ whose kernels
contain the vectors $\{\m_P(x) : x\in \C^n, \ g(x)=0\}$.  Then 
\[\HS(f) \cap L \ \ = \ \ U \HS(h) U^*.  \]
For any matrix $A\in \HS(h)$, the matrix $B = U A U^*$ is positive
semidefinite.  Since
$\m_P^t B \m_P=g \cdot \m_Q^t A \m_Q \cdot \overline{g}=f$, we
conclude that $B$ lies
in $\HS(f)\cap L$.
 
Conversely, suppose that $B$ belongs to $ \HS(f)\cap L$.  Each entry
of the vector $m_P^t B$ vanishes on $\V_{\C}(g)$.  The polynomials
$\{g \cdot x^{\alpha}: \alpha\in Q \}$ form a basis for the space of
polynomials in $\C[\x]_{P}$ vanishing on the variety of $g$.  Indeed,
if $q\in\C[\x]_P$ vanishes on $\V_{\C}(g)$, then $g$ divides $q$,
since $g$ is square-free.  As $\new(q)$ is contained in
$P = \new(g) + Q$, we conclude that $\new(q/g)$ is contained in $Q$.
Therefore, the linear space spanned by the entries of $m_P^t B$ is
contained in ${\rm span}\{g \cdot x^{\alpha}: \alpha\in Q \}$.  This
implies that the column span of $B$ is contained in the column span of
$U$.  We can then write $B$ as $UAU^*$ for some positive semidefinite
matrix $A$.  Since $B$ is a Gram matrix of $f$, $A$ must be a Gram
matrix of $h$.
\end{proof}

This is particularly interesting for Gram spectrahedra of nonnegative binary forms, 
which factor completely into Hermitian squares. 

\begin{Cor}\label{cor:HermGramRankSum}
Let $f \in \R[x_1, x_2]_{2d}$ be a positive binary form with
distinct roots. Then $\HS(f)$ contains $2^d$ matrices of 
rank one. The sum of rank-one matrices $v_1v_1^*, \hdots, v_sv_s^*$ in $\HS(f)$ satisfies
\[
{\rm rank}\left(\sum\nolimits_{k=1}^s v_kv_k^*\right) \ \ \leq \ \ d+1 -
\deg(\gcd(p_1, \hdots, p_s))
\]
where $p_k = m_d^t v_k \in \C[x_1, x_2]_d$ for each $k=1, \hdots, s$. 
For each $2\leq s\leq 2^d$, there are $s$ rank-one
matrices in $\HS(f)$ whose sum has rank at most $\lceil \log_2(s) \rceil +1$.
\end{Cor}

\begin{proof} As in the proof of Proposition~\ref{prop:binaryforms}, rank-one matrices $vv^*\in \HS(f)$ correspond to 
factorizations of $f$ as a Hermitian square $p\overline{p}$, of which there are $2^d$. 

Let $v_1v_1^*, \hdots, v_sv_s^*\in \HS(f)$, write
$p_k = m_d^t v_k \in \C[x_1, x_2]_d$, and let $g$ be their greatest
common divisor. Every root $x$ of $g$ gives a vector $\m_d(x)$ in the
kernel of $\sum_k v_kv_k^*$. Since $f$ has distinct roots, so does
$g$. By the Vandermonde formula, the vectors $m_d(x)$ are linearly
independent, giving the desired rank bound.
For any $e \leq d$, we can factor $f = g\overline{g} \cdot h$, where
$h$ has degree $2e$.  The Hermitian Gram spectrahedron $\HS(h)$ then
contains $2^e$ matrices of rank one, which sum to a matrix of rank $e+1$.  
For $s \leq 2^e$, choose $s$ rank-one matrices in $\HS(h)$. Their images
$v_1v_1^*, \hdots, v_sv_s^*\in \HS(f)$ (under the linear isomorphism
of Proposition~\ref{prop:HermFactor}) have the desired property.
\end{proof}

This allows us to bound the rank of sums of rank-$2$ matrices in the \emph{real
symmetric} Gram spectrahedron of a binary form.  These bounds seem difficult to
understand purely in the language of symmetric Gram matrices.

\begin{Cor}\label{cor:GramRankSum}
Let $f \in \R[x_1, x_2]_{2d}$ be a positive bivariate polynomial with
distinct roots.
Suppose $A_1, \hdots, A_s \in \GS(f)$ have rank 2. 
We can write $A_k = \frac{1}{2}(v_kv_k^* +\overline{v_k v_k^*})$ for each $k=1, \hdots, s$, where $v_kv_k^* \in \HS(f)$. 
Then 
\[
{\rm rank}\left(\sum\nolimits_{k=1}^s A_k \right) \ \leq \ 2d + 2 - 2
\deg(\gcd(p_1, \hdots, p_s)) 
\]
where $p_k = m_d^t v_k \in \C[x_1, x_2]_d$ for each $k=1, \hdots, s$. 
For each $2\leq s\leq 2^{d-1}$, there are $s$ rank-two
matrices in $\HS(f)$ whose sum has rank at most $2\lceil \log_2(s) \rceil +2$.
\end{Cor}

There are some choices inherent in this bound, as we can write $A_k$ both as $\frac{1}{2}(v_kv_k^* + \overline{v_k v_k^*})$ and $\frac{1}{2}(\overline{v_k v_k^*}+v_kv_k^*)$.  
To obtain the best bound, one should
maximize $\deg(\gcd(q_1, \hdots, q_s))$ over all choices of $q_k \in \{p_k, \overline{p_k}\}$ for $k=1, \hdots, s$.

\begin{Exm}
Consider a positive univariate polynomial $f\in \R[t]_{\leq 12}$ with
distinct roots. Its Gram spectrahedron $\GS(f)$ in  $\Sym_7$
has dimension $15$ and $32$ vertices of rank two. 
We can write $f = \prod_{j=1}^6((t-a_j)^2 + b_j^2)$  where $a_j, b_j\in \R$ for each $j=1, \hdots, 6$. 
One expects that a sum of four rank-2 matrices in $\Sym_7$ has full rank 7. However, by Corollary~\ref{cor:GramRankSum}
there are four rank-2 matrices in $\GS(f)$ whose sum has at most rank 6. To construct them, 
take the four representations of $f$ as a Hermitian square, $f = p_1\overline{p_1}=p_2\overline{p_2}=p_3\overline{p_3}= p_4\overline{p_4}$, given by 
\begin{align*}
p_1 = g \cdot  ( t-(a_5 + i b_5)) \cdot ( t-(a_6 + i b_6)) \ \ \ &p_2 = g \cdot  ( t-(a_5 + i b_5)) \cdot ( t-(a_6 - i b_6)) \\
p_3 = g \cdot  ( t-(a_5 - i b_5)) \cdot ( t-(a_6 + i b_6))  \ \ \  &p_4= g \cdot  ( t-(a_5 - i b_5))\cdot ( t-(a_6 - i b_6)),
\end{align*}
where $g = \prod_{j=1}^4 ( t-(a_j + i b_j))$. For $k = 1, \hdots, 4$,
let $v_k\in \C^7$ be the vector of coefficients of $p_k$. 
Each $v_kv_k^*$ is a rank-one Hermitian Gram matrix of $f$, and the sum $\sum_{i=1}^4 v_kv_k^*$  has rank at most 3. 
Also, $A_k = \frac{1}{2}(v_kv_k^* + \overline{v_kv_k^*})$ is a rank-two matrix in $\GS(f)$. The sum $\frac{1}{4}(A_1+A_2+A_3+A_4)$
is the sum of $\frac{1}{8}\sum_{i=1}^4 v_kv_k^*$ and $\frac{1}{8}\sum_{i=1}^4 \overline{v_kv_k^*}$, and thus has 
rank at most $6$. This means that the convex hull of $A_1,
A_2, A_3, A_4$ is contained in the boundary of $\GS(f)$.
\end{Exm}

\section{Sums of squares on products of simplices}\label{sec:Biquadratic}

We end with a short discussion of bi-quadratic forms, which provide fascinating connections between 
sums of squares and questions in topology and complexity. Bi-quadratic forms correspond to 
polynomials with Newton polytope $2P$ where $P$ is the product of two simplices $\Delta_{r-1}\times \Delta_{s-1}$. 
A particularly interesting instance is
\[f_{r,s} \  = \ (x_1^2+ \hdots + x_r^2) (y_1^2+ \hdots + y_s^2).\]
Expansion writes $f_{r,s}$ as a sum of $r\cdot s$ squares, but the sum-of-squares length of $f_{r,s}$ is often smaller 
than $r\cdot s$ and is not known in general. 
For example, $f_{2,2} = (x_1y_1 - x_2y_2)^2 + (x_2y_1+ x_1y_2)^2$. 
This comes from the multiplicativity of the complex norm: $|x| \cdot |y| = |x\cdot y|$ 
where $x=x_1 + ix_2$ and $y = y_1 + iy_2$. 

In general, the length of $f_{r,s}$ 
depends on the existence of \emph{normed bilinear maps}. 
Specifically, if $f_{r,s}$ is a sum of $m$ squares $p_1^2 + \hdots + p_m^2$, then $p = (p_1, \hdots,p_m)$ defines 
a bilinear map $p: \R^r\times \R^s \rightarrow \R^m$ satisfying $|p(x,y)| = |x|\cdot |y|$.
Conversely, a normed bilinear map $p: \R^r\times \R^s \rightarrow \R^m$ gives a representation of $f_{r,s}$ 
as a sum of $m$ squares. 
The existence of these maps is a classical problem in topology dating
back to Hurwitz \cite{Hurwitz}; 
see \cite{MR783137} for a survey of the long, rich history of these problems. 
Radon and Hurwitz  \cite{Hurwitz, Radon} characterized the existence of normed bilinear maps in the case $r=m$. 
The existence is understood for some other values of $(r,s,m)$ (see e.g. \cite{MR1760503}), but is open in general. 

A special case of interest is $r=s$. Let $\mathcal{S}(n)$ denote the length of $f_{n,n}$.
As for $n=2$, multiplicativity of norms on the quaternions and octonions can be used to show that 
$\mathcal{S}(4) = 4$ and $\mathcal{S}(8) = 8$. Hurwitz proved that $n< \mathcal{S}(n)$ for $n\not\in \{1,2,4,8\}$.  
One motivation for Hurwitz's study was to understand the value of $\mathcal{S}(16)$, which is still unknown.  
In \cite{MR1001712}, it is shown that $29 \leq \mathcal{S}(16) \leq 32$. 
More generally, it is known that $n \leq \mathcal{S}(n)\leq(2 -
o(1))n$ holds; see \cite{MR0144355, MR782229}.  

The polynomial $f_{n,n}$ is far from generic in the space of bi-quadratic forms. 
In \cite{HardSOS}, the authors construct an explicit family of bi-quadratic polynomials  with 
Newton polytope $2\Delta_{n-1}\times2\Delta_{n-1}$ whose length is at
least $n^2/16$. 
It would be interesting to understand the lengths of general bi-quadratic forms.

\begin{Ques}
What is the range of typical lengths for sums of squares with Newton polytope $\Delta_{n-1}\times\Delta_{n-1}$? 
What is the generic length over $\C$?
\end{Ques}

The lengths of bi-quadratic forms over $\C$ have applications in computational complexity. 
Hrube\v{s}, Wigderson, and Yehudayoff describe how complex sum-of-squares lengths provide lower bounds on the ``non-commutative circuit size'' of the $n\times n$ permanent \cite{HardSOS}.
Results like these demonstrate the wide ranging connections of sums of squares to other fields in mathematics 
and deserve further investigation.

\end{document}